\documentclass[12pt, a4paper, reqno]{amsart}
\usepackage{amsmath, amsfonts, amssymb}
\usepackage[active]{srcltx}
\usepackage{latexsym}
\usepackage{eucal}
\usepackage{enumerate}
 \usepackage{float}
\usepackage{mathrsfs} 

\allowdisplaybreaks

\usepackage{cite}

    \newtheorem{theorem}{Theorem}[section]
    \newtheorem{lemma}[theorem]{Lemma}
    \newtheorem{remark}[theorem]{Remark}
    \newtheorem{example}[theorem]{Example}

    \newcommand{\ct}[1]{\langle {#1}\rangle \lower.3ex\hbox{$_{t}$}}
    \newcommand{\lt}[1]{[ {#1}] \lower.3ex\hbox{$_{t}$}}

\newcommand{\R}{\mathbb{R}}

\theoremstyle{definition}
\newtheorem{definition}[theorem]{Definition}

\newcommand{\abs}[1]{\lvert#1\rvert}

\newcommand{\beq}{\begin{equation}}
\newcommand{\bea}[1]{\begin{array}{#1} }
\newcommand{\eeq}{ \end{equation}}
\newcommand{\ea}{ \end{array}}
\newcommand{\ep}{\eps}

\newcommand{\om}{\omega}
\newcommand{\Om}{\Omega}

\newcommand{\ud}{\, d}
\newcommand{\half}{{\frac{1}{2}}}

\newcommand{\ol }{\overline}

\newcommand{\tr}{\operatorname{trace}}
\newcommand{\eps}{\varepsilon}
\newcommand{\diag}{\operatorname{diag}}

\newcommand{\kom}[1]{}
\renewcommand{\kom}[1]{{\bf [#1]}}
\newcommand{\Rn}{\mathbb{R}^n}

\newcommand{\ac}{\mathcal{AC}}
\newcommand{\str}{\mathcal{S}}

\begin{document}
\title[Tug-of-war, market manipulation and option pricing]{Tug-of-war, market manipulation\\ and option pricing  }

\author{K. Nystr{\"o}m, M. Parviainen }

\address{Kaj Nystr\"{o}m, Department of Mathematics, Uppsala University\\
S-751 06 Uppsala, Sweden}
\email{kaj.nystrom@math.uu.se}
%
%
\address{Mikko Parviainen, Department of Mathematics and Statistics, University of
Jyv\"askyl\"a, PO~Box~35, FI-40014 Jyv\"askyl\"a, Finland}
\email{mikko.j.parviainen@jyu.fi}

\thanks{\emph{Acknowledgements.} The authors are grateful to two anonymous referees and the managing editor for their careful reading of the paper and their valuable comments. This work was partly done during a visit to the Institut Mittag-Leffler. MP is supported by the Academy of Finland.}

\begin{abstract}
We develop an option pricing model based on a tug-of-war game. This two-player zero-sum stochastic differential game is formulated in the context of a multi-dimensional financial market. The issuer and the holder  try to manipulate
asset price processes in order to minimize and maximize the expected discounted reward. We prove
that the game has a value  and that the value function is the unique viscosity solution to a terminal value problem for a parabolic  partial differential equation
involving the  non-linear and completely degenerate  infinity Laplace operator.
\end{abstract}

\maketitle

\setcounter{equation}{0} \setcounter{theorem}{0}

\section{Introduction}

A feature of illiquid markets is that large transactions move prices. This is a disadvantage
for traders needing to liquidate large portfolios, but there are also situations where traders may benefit from
moving prices. For example, a trader holding a large number of options may have an
incentive to impact the dynamics of the underlying and to move the option
value in a favorable direction if the increase in the option value outweighs the trading
costs in the underlying. There is some empirical evidence, see \cite{GS}, \cite{P}, \cite{KS}, that in illiquid markets, option traders are in fact able to increase the value of a derivative by
moving the price of the underlying.

 We consider option pricing in the context of a two-player zero-sum stochastic differential game in a multi-dimensional financial market. In the game, the issuer and holder of the option try, respectively, to manipulate/control the drifts and the volatilities of the
assets in order to minimize and maximize, respectively, the expected discounted reward at the terminal date $T$. An important contribution is that we establish a connection between option pricing and a tug-of-war game. In the prevailing model for option pricing, the governing partial differential equation is the Black-Scholes equation. In our context, the partial differential equation becomes substantially more involved due to the presence of the non-linear and completely degenerate infinity Laplace operator.

\subsection{Price dynamics}
\label{sec:price-dyn}
First we  give a heuristic description of the price formation process. Let
\begin{eqnarray}\label{pp1-}
S(t)=(S_1(t,\omega),...,S_n(t,\omega)):[0,T]\times\Omega\to \mathbb R^n_+\notag
\end{eqnarray}
be the stochastic process which represents
the prices  of $n$ assets at time $t\in[0,T]$. To keep mathematical tractability, we formulate the dynamics of $S=S(t)$ as a system of stochastic differential equations for the vector of $\log$-returns
\begin{eqnarray*}
  X(t)&=&(X_1(t,\omega),...,X_n(t,\omega)):[0,T]\times\Omega\to \mathbb R^n,\\
  X_i(t)&=&\log(S_i(t))\quad \text{for}\quad  i\in\{1,...,n\}.
\end{eqnarray*}
Let $(\Omega,\mathcal{F},\mathbb P)$ be a probability space satisfying the standard assumptions and let $\{\xi_{i,k}\}$, for $i\in\{1, \ldots,n,n+1\}$, $k\in\mathbb N$, be sequences of  i.i.d random variables such that $$\mathbb P(\xi_{i,k}=1)=1/2=\mathbb P(\xi_{i,k}=-1).$$ In particular, each $\xi_{i,k}$ represents the outcome  of a standard coin toss. We let
$\mathcal{F}_k$ be a filtration of $\mathcal{F}$ to which $\{\xi_{i,k}\}$  are adapted.

Let $N\in \mathbb N$ denote our discretization parameter.
We let $X_{i,k}^N$ denote the state of
the log-returns of asset $i$ after step $k$. At this level, the model can be expressed as
\begin{eqnarray}\label{model1-}
X_{i,k}^N-X_{i,k-1}^N&=&\mbox{a random walk increment with drift}\notag\\
&&+\mbox{ an increment resulting from price manipulation}\notag\\
&&\ \ \mbox{ modeled as a tug-of-war game}.\notag
\end{eqnarray}

To be more precise, at step $k$ of the game, a sample of $(\xi_{1,k},\ldots,\xi_{n+1,k})$ is generated.
For component $i\in \{1\ldots,n\}$, the contribution from the random walk  with drift is modeled as
\[
\begin{split}
\frac {\mu_i}N+\frac 2{\sqrt{N}}\sigma_i\xi_{i,k}
,\end{split}
\]
where $\mu_i\in\mathbb R$ and $\sigma_i>0$ represents the magnification of the step $2\xi_{i,k}$. Volatility can vary from the asset to asset.

Next,  let
$\{\theta_{k}^\pm\}=\{(\theta_{1,k}^\pm,\ldots,\theta_{n,k}^\pm)\}$ be $\mathcal{F}_k$-adapted random variables such that $\theta_{k}^\pm\in \{x\in\mathbb R^n:\ |x|\leq1/\sqrt{N}\}$. The
sequences $\{\theta_{k}^+\}$, $\{\theta_{k}^-\}$, correspond to the control actions of the maximizing and  minimizing player in a tug-of-war game.  In particular, it is assumed  that each of the two players can affect the price process and push it in a  favorable direction but turns are taken randomly. In this setting,  the increment of component $i$, at step $k$,  based on coin toss $\xi_{n+1,k}$, is
\[
\begin{split}
2\sigma_i&\biggl (\theta_{i,k}^+\frac {(1+\xi_{n+1,k})}2+\theta_{i,k}^-\frac {(1-\xi_{n+1,k})}2\biggr )\\
&=2\sigma_i\biggl (\frac {(\theta_{i,k}^+-\theta_{i,k}^-)}2\xi_{n+1,k}+\frac {(\theta_{i,k}^++\theta_{i,k}^-)}2\biggr ).
\end{split}
\]
Again the actions of the players are magnified by the factors $\{\sigma_i\}$.

Put together, the positions of the log-returns after $j$ steps are   $X_{j}^N=(X_{i,j}^N,...,X_{i,j}^N)$, where
\[
\begin{split}
X_{i,j}^N=&x_i+\mu_{i}\frac j N+\frac 2{\sqrt{N}}\sigma_i\sum_{k=1}^j\xi_{i,k}\notag\\
&+2\sigma_i\sum_{k=1}^j\biggl (\frac {(\theta_{i,k}^+-\theta_{i,k}^-)}2\xi_{n+1,k}+\frac {(\theta_{i,k}^++\theta_{i,k}^-)}2\biggr ).\notag
\end{split}
\]

We define $\{W_i^{N}(t)\}_{t\geq 0}$, $i\in\{1,...,n+1\}$, by setting $$(W_1^{N}(0),\ldots,W_{n+1}^N(0))=0$$ and using the relations
\begin{eqnarray}\label{model2}
W_i^{N}(t)&=&W_i^{N}((k-1)/N)+\biggl (t-\frac {k-1}N\biggr)\sqrt{N}\xi_{i,k},\notag
\end{eqnarray}
whenever $t\in((k-1)/N,k/N],\ k\in\mathbb N$. Moreover, we define continuous time processes by setting
$$
X^N(t)=X^N_{[Nt]}, \quad \theta^{\pm,N}(t)=\sqrt{N}\theta^{\pm}_{[Nt]}.
$$
With this notation, the above dynamics becomes
\begin{eqnarray}\label{model3}
X_{i}^N(t)=A_{i}^N(t)+B_{i}^N(t)+C_{i}^N(t),\qquad i\in \{1,\ldots,n\},
\end{eqnarray}
where
\begin{eqnarray}\label{model3+}
A_{i}^N(t)&=&x_i+\int_0^t\mu_ids+\int_0^t\sigma_idW_i^{N}(s),\notag\\
B_{i}^N(t)&=&\sigma_i\int_0^t(\theta_i^{+,N}(s)-\theta_i^{-,N}(s))d W_{n+1}^{N}(s),\notag\\
C_{i}^N(t)&=&\sigma_i\int_0^t\sqrt{N}(\theta_i^{+,N}(s)+\theta_i^{-,N}(s))ds.\notag
\end{eqnarray}
Then, by passing to the limit, using Donsker's invariance principle,
\begin{eqnarray}\label{model3++}
A_{i}^N(t)&\to&x_i+\int_0^t\mu_ids+\int_0^t\sigma_idW_{i}(s),\ \notag\\
B_{i}^N(t)&\to& \sigma_i\int_0^t(\theta^{+}_i(s)-\theta^{-}_i(s))d W_{n+1}(s),\notag
\end{eqnarray}
as $N\to\infty$ where $W_i$,  $W_{n+1}$, are standard and independent Brownian motions. To understand the continuous time limit
of the outlined price dynamics, as  $N\to\infty$, the key difficulty is to understand the asymptotic behavior of the term $C_{i}^N(t)$. A solution, due to \cite{atarb10} in the context of time independent equations, is to replace $\sqrt{N}$ with dynamically controlled quantities $d^+$ and $d^-$. This approach is  motivated by the connection between tug-of-war games
and the infinity Laplace operator in \cite{PSSW}. Therefore, as described later, the core of the model is given by
\begin{equation}
\label{eq:dynamics}
\begin{split}
dX_i(s)=&\Big(\mu_i+\sigma_i(d^+(s)+d^-(s))(\theta_i^+(s)+\theta^-_i(s))\Big)ds\\
&+\sigma_idW_i(s)+ \sigma_i (\theta_i^+(s)-\theta^-_i(s))dW_{n+1}(s),
\end{split}
\end{equation}
with sufficient assumptions on the controls $d^\pm$, $\theta^\pm$.

\subsection{Fair game option pricing} Let $(\Omega,\mathcal{F},\{\mathcal{F}_t\},\mathbb P)$ denote a complete filtered probability space with a right-continuous filtration supporting an $(n+1)$-dimensional and $\{\mathcal{F}_t\}$-adapted Brownian motion $W=(W_1,...,W_{n+1})$. We assume that all components are independent.

There are  two competing players, one maximizing and one minimizing, which both attempt to control and manipulate the log-returns  $X(t)$ of the underlying assets. Denote by $\mathbb S^{n-1}$ the unit sphere of $\Rn$. We let
 $$
 \mathcal{H}:=\mathbb S^{n-1}\times [0,\infty),
 $$
 and
\begin{eqnarray}
A^+:=A^+(t):=(\theta^+(t),d^+(t)),\qquad  A^-:=A^-(t):=(\theta^-(t),d^-(t)),\notag
\end{eqnarray}
where
\begin{eqnarray}
\theta^\pm(t)\in\mathbb S^{n-1},\ d^\pm(t)\in[0,\infty),\ t\in[0,T],\notag
\end{eqnarray}
are $\{\mathcal{F}_t\}$-adapted stochastic processes representing the control actions of the maximizing and minimizing player. Heuristically,  $\theta^\pm(t)$ denote the directions and $d^\pm(t)$ the lengths of the steps taken by the players. ${\ac}$ denotes the set of all admissible controls.
  Each player also chooses a strategy $\rho^\pm$, which represents a response to the actions of the opponent, i.e.\ the strategies $\rho^\pm$ are functions from the space of controls to the space of controls. $\mathcal{S}$ denotes the set of all admissible strategies.
  Detailed definitions of (admissible) controls (${\ac}$) and strategies ($\mathcal{S}$) are given below in Definitions \ref{admissiblecont} and \ref{strategy}. Using this notation, the dynamics of the log-returns is given by \eqref{eq:dynamics}.

  Note that in \eqref{eq:dynamics}, the time-dependent controls of the players enter in the drift coefficient, and in the diffusion coefficient of the one-dimensional Brownian motion $W_{n+1}$. Hence, this part of the dynamics is degenerate in the
sense that it is possible for the players to completely switch off the one-dimensional Brownian motion  $W_{n+1}$.

Given  $A^\pm=(\theta^\pm,d^\pm)$ and a pay-off function $g$ at $T$,  we set
 \begin{eqnarray}\label{par1+}
J^{(x,t)}(A^+,A^-)&:=&\mathbb E[e^{-r(T-t)}g(X^{(x,t)}(T))]\notag\\
&=&\mathbb E[e^{-r(T-t)} g(X(T))]
\end{eqnarray}
where the superscript ${(x,t)}$ indicates that the game starts at position $x$ at time $t$. The expectation $\mathbb E[\cdot]$ is taken with respect to the measure $\mathbb P$.
\begin{definition}\label{game} The upper and lower values of the stochastic dynamic game, denoted
$U^+(x,t)$ and $U^-(x,t)$, are defined
 \[
\begin{split}
U^+(x,t)&=\sup_{\rho^+\in{\str}}\inf_{A^-\in {\ac}}J^{(x,t)}(\rho^+(A^-),A^-),\\
U^-(x,t)&=\inf_{\rho^-\in{\str}}\sup_{A^+\in {\ac}}J^{(x,t)}(A^+,\rho^-(A^+)).
\end{split}
\]
The game is said to have a value  at $(x,t)$ if $U^+(x,t)=U^-(x,t)$. If $U^+(x,t)=U^-(x,t)=:U(x,t)$, then we say that $U(x,t)$ is the fair game value of the option.
\end{definition}

\subsection{Statement of main results} Our fair game value of the option is related to the degenerate partial differential operator $F$
\begin{equation}
\begin{split}
\label{par5qq+}
F(u,Du,D^2u):=&\frac 2{|Du|^2}\Big(\sum_{i,j=1}^n u_{x_ix_j}u_{x_i}u_{x_j}\sigma_i\sigma_j\Big )\\
&+\frac 12 \Big(\sum_{i=1}^n \sigma_i^2u_{x_ix_i}\Big )+\sum_{i=1}^n\mu_iu_{x_i}-ru,
\end{split}
\end{equation}
where
$
Du:=(u_{x_1},....,u_{x_n})^\prime,$  and
$D^2u$ is the matrix consisting of the second order derivatives.  We consider the terminal value problem
\begin{equation}
\label{eq:F-eq}
\begin{split}
\begin{cases}
 \partial_t u+F(u,Du,D^2u)=0,&\text{in }\mathbb R^n\times (0,T), \\
 u(x,T)=g(x),& \text{on }  \mathbb R^n.
\end{cases}
\end{split}
\end{equation}
Solutions to \eqref{eq:F-eq} have to be understood in the viscosity sense (see Definition \ref{vis1} and \ref{def:visc-for-F}). Concerning $g$, we adopt the following convention: throughout the paper it is our standing assumption that the function $g$ is a positive bounded Lipschitz function, i.e.
\begin{equation}
\label{eq:standing}
\begin{split}
\sup_{x\in\mathbb R^n}g(x)+\sup_{x,y\in\mathbb R^n,x\neq y}\frac {|g(x)-g(y)|}{|x-y|}\leq L,
\end{split}
\end{equation}
for some $L<\infty$. The assumptions concerning boundedness and positivity of $g$ are only imposed to minimize additional technical difficulties. The main result of the  paper is the following.
\begin{theorem}\label{th2} Let $g$ be as in \eqref{eq:standing} and let $U^\pm$ be the upper and lower values of the stochastic dynamic game as in Definition \ref{game}. Then, \begin{eqnarray*}
U^+\equiv U^-\quad \mbox{on}\quad  \mathbb R^n\times[0,T]
\end{eqnarray*}
and $U:= U^+\equiv U^-$ is the unique viscosity solution to  \eqref{eq:F-eq}. In particular, $U(x,t)$ is the fair game value of the option. 
  \end{theorem}

  \begin{remark} In the original Black-Scholes model, the arbitrage free price of a simple European contract is, after changing to log-returns, the unique solution to the Cauchy problem for a linear second order uniformly parabolic equation (heat equation). In contrast, in our context, the underlying PDE is much more involved due to the presence of the non-linear and  completely degenerate infinity Laplace operator. Naturally, the fair game value $U$ is different from the value produced through arbitrage free option pricing.
  \end{remark}

    \begin{remark} \label{inho} Note that our set-up may seem unrealistic: it should be costly
for the players to influence the asset prices. However, as discussed below, we can introduce a running cost  to the model and  Theorem \ref{th2}. In this case, the fair game value of the option will be the unique viscosity solution to the non homogenous problem \eqref{eq:F-eq--} below. However, for simplicity of exposition, we only prove Theorem \ref{th2} without a running cost.
  \end{remark}

  \begin{example}\label{payoff1} Consider $g:\mathbb R^n\to\R$  defined by
\begin{eqnarray}\label{bound+}
g(x_1,...,x_n)=\max\{K-w_1e^{x_1}-....-w_ne^{x_n},0\}
\end{eqnarray}
where $w_1+....+w_n=1$, $w_i\geq 0$, and where $K$, the strike, is a positive real number. Then $g$ represents a put option written on the
index $$w_1S_1(T)+....+w_nS_n(T)=w_1e^{x_1}+....+w_ne^{x_n}$$ with strike $K$. Obviously $g$ satisfies \eqref{eq:standing}
for some $L<\infty$.
\end{example}

\subsection{Relation to the literature and economic relevance} By definition, option pricing is devoted to the valuation of options and other contingent claims, and, since the pioneering papers of Black-Scholes \cite{BS} and Merton \cite{Me}, it has also found its way into many
applications beyond finance.

To put our model into context, it is relevant to first recall that the models in \cite{BS}, \cite{Me} are derived based on stylized assumptions concerning the market. However,  they are based on fundamental economic principles.


While the Black-Scholes model is derived based on the idea of no arbitrage, our model is based on the idea that, in addition, the price is influenced by a tug-of-war between two players: the issuer and the holder of the option. As such, our model
touches on the mathematical modeling of illiquid financial markets and market manipulation, and the use of stochastic differential games for option pricing.

A substantial part of the literature on illiquid financial markets
focuses on either optimal hedging and portfolio liquidation strategies for a single large
investor with market impact, see for example \cite{CJP}, \cite{AFS}, predatory trading, see \cite{BP},\cite{CLV}, \cite{SS}, or the problem
of market manipulation using options, see \cite{J}, \cite{KS}. For example, in \cite{J}, it is shown that by introducing derivatives into an
otherwise complete and arbitrage-free market, certain manipulation strategies may appear for a large
trader.  The strategic interaction between large investors
and implications for market microstructure are discussed in \cite{K}, \cite{FoS}, \cite{BCW}, \cite{CV}. The papers \cite{BP}, \cite{CLV}, \cite{SS},
consider predatory trading, where liquidity providers try to benefit from the liquidity
demand of large investors.

For zero sum stochastic differential games, we  refer to \cite{flemings89}, \cite{HL}, and for nonzero sum games to \cite{F}, \cite{BCR}, see also \cite{N}, \cite{BL}. In general, there seems to exist only a limited literature  devoted to the integration
of game theoretical aspects, like strategic financial decisions of agents, into continuous time frameworks. We are only aware of two lines of research in this direction.

One line  seeks to account for  model uncertainty using control theory. The idea
is to  introduce optimal and risk-free strategies in security markets, based on which traders try to meet their obligations when, for instance, the volatility is unknown or uncertain. This is an alternative approach, which in the end is deterministic, to option pricing and contingent claim hedging.  In \cite{Ly}, the author develops a robust control model, in multi-dimensional markets without friction, assuming that the volatility is unknown and only assumed to lie in some convex region depending on the prices of the underlying securities and time, see also \cite{ALP95} and  \cite{AvP96}. In \cite{Ly}, the PDEs associated to the control problems considered are fully non-linear parabolic equations of Pucci-Bellman type. This approach to option pricing is also developed in \cite{Mc}, where the author proves  that option prices in standard models can be characterized as viscosity solutions of the corresponding Hamilton-Jacobi equations. Similar robust control problems are also considered and developed in \cite{Be}, see the references in \cite{Be} for an overview of contributions in this direction, and \cite{Ko}.

 Another line, which appears in \cite{Z}, develops a method which decouples into three steps: $(i)$, the formulation of a game among players; $(ii)$, the valuation of future uncertain payoffs using (standard) option pricing theory; $(iii)$, the resolution of the game for the optimal strategies, starting from the last decision to be made, using backward induction or alternative methods. A strength of the approach in \cite{Z} is  that $(ii)$ and $(iii)$ are separated steps,  enabling the integration of established methods for each step.

Our model and this paper represent  a new and different line of research devoted to fair game option pricing and intuitively appealing stochastic differential games modeled as tug-of-war games. Within our model, we see at least three areas for further exploration. First, the fair game value of the option is derived based on principles different from no arbitrage and it is an interesting problem to understand the relation between these two different approaches. Second, while this paper is mainly of theoretical nature, it is important to study \eqref{eq:F-eq} from a numerical point of view. This may require novel methods due to the presence of the non-linear and completely degenerate infinity and parabolic infinity Laplace operator. Third, adding American, Asian or other features to the underlying derivative,  our setup paves the way for extensive studies of completely
new PDE problems related to tug-of-war games.

As discussed in Remark \ref{inho}, it may seem unrealistic that the players can influence the asset price without a cost. However, this can be incorporated by introducing a running cost. Let $h:\mathbb R^n\times [0,T]\to\mathbb R$ and assume that there exists $\alpha>0$ such that $h(x,t)\leq-\alpha$ for all $(x,t)\in \mathbb R^n\times [0,T]$. Given $A^\pm=(\theta^\pm,d^\pm)$, $h$, and  $g$ at $T$,  let
 \begin{eqnarray}\label{par1+}
\tilde J^{(x,t)}(A^+,A^-)&:=&\mathbb E\biggl[\int_t^Te^{-r(T-s)}h(X^{(x,t)}(s),s)\, ds\notag\\
&&\quad\quad\quad\quad+e^{-r(T-t)}g(X^{(x,t)}(T))\biggr]
\end{eqnarray}
where, again, the superscript ${(x,t)}$ indicates that the game starts at position $x$ at time $t$. Let $\tilde U^\pm$ be defined as in Definition \ref{game}, but based on $\tilde J$ instead of $J$. Then, see also Remark \ref{inho}, Theorem \ref{th2} can be generalized to this situation and, under sufficient assumptions,
\begin{eqnarray*}
\tilde U^+\equiv \tilde U^-\quad \mbox{on}\quad  \mathbb R^n\times[0,T].
\end{eqnarray*}
Furthermore,  $\tilde U:= \tilde U^+\equiv \tilde U^-$ is the unique viscosity solution to  the non homogenous problem
\begin{equation}
\label{eq:F-eq--}
\begin{split}
\begin{cases}
 \partial_t u+F(u,Du,D^2u)=-h(x,t),&\text{in }\mathbb R^n\times (0,T), \\
 u(x,T)=g(x),& \text{on }  \mathbb R^n.
\end{cases}
\end{split}
\end{equation}
 In this case, $\tilde U(x,t)$ can be referred to as the fair game value of the option accounting for the cost of influencing the asset price (transaction cost).

\subsection{Brief outline of the proof of Theorem \ref{th2}} The difficulty encountered when proving Theorem \ref{th2} stems from the unboundedness
of controls and strategies and from the potential degeneracy of the underlying dynamics. To overcome the unboundedness of the action sets, we first approximate the original stochastic differential game by a sequence of games with bounded controls ${\ac}_m$ and bounded strategies ${\str}_m$, with bounds tending to $\infty$ as $m\to\infty$. The upper and lower values of the associated stochastic dynamic games are defined by
\begin{eqnarray*}
&&U_m^+(x,t)=\sup_{\rho^+\in{\str}_m}\inf_{A^-\in {\ac}_m}J^{(x,t)}(\rho^+(A^-),A^-),\\
&&U_m^-(x,t)=\inf_{\rho^-\in{\str}_m}\sup_{A^+\in {\ac}_m}J^{(x,t)}(A^+,\rho^-(A^+)),
\end{eqnarray*}
where $J^{(x,t)}$ is given in \eqref{par1+}.
The upper and lower values are unique.
An important step is to connect the value functions to viscosity solutions to the following terminal value problems involving Bellman-Isaacs type equations:
\begin{eqnarray}\label{tron1}
\partial_t u-H_m^+(u,Du,D^2 u)&=&0\ \ \ \ \ \ \ \mbox{ in }\mathbb R^n\times (0,T),\notag\\
u(x,T)&=&g(x)\ \ \ \mbox{ on }\mathbb R^n,
\end{eqnarray}
\begin{eqnarray}\label{tron2}
\partial_t u-H_{m}^-(u,Du,D^2 u)&=&0\ \ \ \ \ \ \ \mbox{ in }\mathbb R^n\times (0,T),\notag\\
u(x,T)&=&g(x)\ \ \ \mbox{ on }\mathbb R^n.
\end{eqnarray}
The operators $H_m^\pm$ are introduced later and, here, we simply note that the equations in \eqref{tron1}, \eqref{tron2}, are non-linear parabolic equations and that these equations are the relevant Bellman-Isaacs equations associated to our problem in the case of bounded controls. Comparison principles and uniqueness of viscosity solutions to \eqref{tron1}, \eqref{tron2} follow along the lines of Giga, Goto, Ishii and Sato, see \cite{gigagis91}, by using doubling of variables as well as the theorem of sums.
Existence is established by the construction of appropriate barriers and by the use of Perron's method, see Section 2.

 In Lemma \ref{lem:value}, we
prove that the unique solutions to  \eqref{tron1}, \eqref{tron2},  $u_m^\pm$, satisfy
\begin{eqnarray}\label{tron3}
u_m^+=U_m^+,\ u_m^-=U_m^-.
\end{eqnarray}
In other words,  the unique solutions to stated terminal value problems produce the upper and lower values of the associated stochastic games. The proof uses Ito's formula and estimates for stochastic differential equations.

To continue, we prove in Lemma \ref{lem:approxham-}  that
\[
\begin{split}
&H_m^\pm\to -F \qquad \mbox{ as }\quad m\to\infty.
\end{split}
\]
To complete the proof of  Theorem \ref{th2}, a key step
is to  prove that there exists $m_0\in\mathbb \{1,2,\ldots\}$, such that the families
$$\{u_m^\pm:\ m\geq m_0\}$$
 are equicontinuous (Lemma \ref{lem:equicontinuity}). The proof is based on a barrier argument. These results enable us to conclude by the Arzel\`a-Ascoli theorem, see Lemma \ref{lemimportant}, that there exists a continuous function $u$ such that
\begin{equation}
\label{eq:sol-convergence}
\begin{split}
u_m^\pm(x,t)\to u(x,t),
\end{split}
\end{equation}
and that the limit $u$ is the unique solution to \eqref{eq:F-eq}.

Finally,  at the end of Section \ref{sec:limits}, we prove Theorem \ref{th2} by showing that, as the bounds on the controls increase, then a subsequence of corresponding value functions converge to a value function for the game with unbounded controls. This, \eqref{tron3} and \eqref{eq:sol-convergence} yield the result.

Our approach is influenced by the works of Swiech \cite{swiech96} as well as Atar and Budhiraja \cite{atarb10}. A different approach to stochastic games is due to Fleming and Souganidis \cite{flemings89}, see also \cite{BL}. Indeed, the approach in \cite{flemings89} is based on establishing a dynamic programming principle based on careful approximation arguments, working directly with the value functions.  The authors then prove that the value functions also solve the associated Bellman-Isaacs equations. However, our model contains degenerate diffusion and unbounded controls, and our approach relies on viscosity theory for non-linear and degenerate partial differential equations already from the beginning. In particular, instead of establishing a dynamic programming principle for the value function,  we show, as explained above, that the unique viscosity solution to the corresponding partial differential equation satisfies a dynamic programming principle.

Our work is  developed based on the recently established connections between discrete time tug-of-war games and infinity harmonic functions  \cite{PSSW}, and tug-of-war games with noise in the context of $p$-harmonic functions \cite{PS}. Here, we also mention the approach based on non-linear mean value formulas developed in \cite{manfredipr10c} and \cite{manfredipr12}. Continuous time stochastic differential games and infinity harmonic functions  were considered in \cite{atarb10}, and \cite{atarb11}. The equation considered in this paper coincides, modulo the presence of the model related constants and a change of the time direction,  with the normalized $p$-Laplace operator considered in \cite{manfredipr10c}, in connection with normalized $p$-parabolic equations and tug-of-war games, see also \cite{banerjeeg12a} and \cite{does11}. The parabolic equation involving a normalized infinity Laplacian is studied in \cite{juutinenk06}.

\setcounter{equation}{0} \setcounter{theorem}{0}

\section{Preliminaries}

Recall that   $(\Omega,\mathcal{F},\{\mathcal{F}_s\},\mathbb P)$  denotes a complete filtered probability space with a right-continuous filtration supporting a $(n+1)$-dimensional and $\{\mathcal{F}_s\}$-adapted Brownian motion $W=(W_1,...,W_{n+1})$. We assume that all components are standard independent Brownian motions.
\begin{definition}[Controls]
\label{admissiblecont} Let
$$
A:=A(s):=(\theta(s),d(s))
$$
 be a progressively measurable stochastic process on $(\Omega,\mathcal{F},\{\mathcal{F}_s\},\mathbb P)$ taking values in
$ \mathcal{H}=\mathbb S^{n-1}\times [0,\infty)$. Then $A$ is called a control. We set $$\Lambda:={\Lambda}(A):=\sup_{\om\in \Om} \sup_{s\in [0,T]} d(s,\omega) \in [0,\infty].$$ The control $A$ is said to be admissible provided  $\Lambda<\infty$, and we denote the set of all admissible controls by  $\mathcal{{AC}}$. 
\end{definition}

\begin{definition}[Strategies]
\label{strategy} A strategy is a mapping
$$
\rho:\ac \to \ac
$$ such that if
$$\mathbb P(A(s)=\tilde A(s)\mbox{ for a.e. }s\in[0,\tau])=1\quad\mbox{and}\quad \Lambda(A)=\Lambda(\tilde A)$$
then
$$\mathbb P(A^\prime (s)=\tilde A^\prime (s)\mbox{ for a.e. }s\in[0,\tau])=1\quad\mbox{and}\quad\Lambda(A^\prime) =\Lambda(\tilde A^\prime)$$
for every $\tau\in[0,T]$, where $A':=\rho(A),\  \tilde A':=\rho(\tilde A).$ Given a strategy $\rho$, we set $${\Lambda}(\rho):=\sup_{A\in \ac}{\Lambda}(\rho(A))\in[0,\infty].$$
A strategy is said to be admissible provided ${\Lambda}(\rho)<\infty$, and we denote the set of all admissible strategies by  $\str$.
\end{definition}

In general, below we will consider controls which depend on the current location $X(t)$. Furthermore, for brevity, we drop the word 'admissible' from now on.
Later, we will approximate controls and strategies by uniformly bounded ones.
\begin{definition}\label{strategyb}
For $m\in \{1,2,\ldots\}$, we define
\begin{eqnarray}\label{par1}
{\ac}_m&:=&\{A\in\ac:\ {\Lambda}(A)\leq m\},\notag\\
{\str}_m&:=&\{\rho\in \str:\ \Lambda(\rho)\leq  m\}.
\end{eqnarray}
\end{definition}

\begin{definition}\label{game+} The upper and lower values of the underlying stochastic dynamic game, with controls in
 ${\ac}_m$ and strategies in ${\str}_m$, are defined as
\begin{eqnarray*}
&&U_m^+(x,t)=\sup_{\rho^+\in{\str}_m}\inf_{A^-\in {\ac}_m}J^{(x,t)}(\rho^+(A^-),A^-),\\
&&U_m^-(x,t)=\inf_{\rho^-\in{\str}_m}\sup_{A^+\in {\ac}_m}J^{(x,t)}(A^+,\rho^-(A^+)).
\end{eqnarray*}
\end{definition}

\subsection{Bellman-Isaacs equations with bounded action sets: viscosity solutions} Let $\Sigma=\diag(\sigma_1,\ldots,\sigma_n)$ and let $\mathcal{M}(n)$ denote the set of all symmetric $n\times n$-dimensional matrices. Given a matrix, or vector $M$, we let $M'$ denote the transpose of $M$. We define
$\Phi:\mathbb S^{n-1}\times\mathbb S^{n-1}\times\mathbb R_+\times\mathbb R_+\times\mathbb R^n\times\mathcal{M}(n)\to \mathbb R$ through
\[
\begin{split}
\Phi(\theta^+&,\theta^-,d^+,d^-,p,M)\\
=&-\half (\theta^+-\theta^-)'\Sigma M\Sigma (\theta^+-\theta^-)\\
&-\half \tr(\Sigma^2 M)-(d^++d^-)(\theta^++\theta^-)\cdot p-\mu\cdot p.
\end{split}
\]

 Let $\mbox{LSC}(\mathbb R^n\times [0,T])$ be the set of lower semi-continuous functions, i.e.\ all functions
 $$f: (\R^n\times [0,T]) \to \R\cup \{\infty\}$$ such that
\begin{equation*}
\liminf_{(y,s)\to (x,t)} f(y,s) \geq f(x,t).
\end{equation*}
Let
$\mbox{USC}(\mathbb R^n\times [0,T])$ be the set of upper semi-continuous functions, i.e. all functions
$$
f: (\mathbb R^n\times [0,T])\to \R\cup\{-\infty\}$$ such that
\begin{equation*}
\limsup _{(y,s)\to (x,t)} f(y,s) \leq f(x,t).
\end{equation*}
We define $\mbox{LSC}_l (\mathbb R^n\times [0,T])$ to consist
of functions $h\in \mbox{LSC}(\mathbb R^n\times [0,T])$ which satisfy the (linear) growth condition
\begin{eqnarray} \label{eq:growth}\label{eq:growth-bound}
|h(x,t)| \leq  c(1+|x|)
\end{eqnarray}
and for some $c\in [1,\infty)$. The space $\mbox{USC}_l (\mathbb R^n\times [0,T])$ is defined analogously.
Furthermore,
\[
\begin{split}
\mbox{C}_l(\mathbb R^n\times [0,T])=\mbox{USC}_l (\mathbb R^n\times [0,T])\cap \mbox{LSC}_l (\mathbb R^n\times [0,T]).
\end{split}
\]
In addition, ignoring $t$ we define $\mbox{C}_l(\mathbb R^n)$ by analogy.
Finally, $C^{1,2}(\mathbb R^n\times [0,T])$ denotes the space of functions that are once continuously differentiable in time and twice continuously differentiable in space.

Given $m\in\{1,2,\ldots\}$, we let
\[
\begin{split}
\mathcal{H}_m=\{(\theta,d)\in\mathcal{H}\,:\, d\leq m\},
\end{split}
\]
and we define $\tilde H_m^+,\tilde H_m^-:\mathbb R^n\times\mathcal{M}(n)\to \mathbb R$ through
\[
\begin{split}
\tilde H^+_m(p,M)=&\sup_{(\theta^-,d^-)\in\mathcal{H}_m}\ \ \inf_{(\theta^+,d^+)\in\mathcal{H}_m}\Phi(\theta^+,\theta^-,d^+,d^-,p,M),\\
\tilde H^-_m(p,M)=&\inf_{(\theta^+,d^+)\in\mathcal{H}_m}\ \ \sup_{(\theta^-,d^-)\in\mathcal{H}_m}\Phi(\theta^+,\theta^-,d^+,d^-,p,M).
\end{split}
\]
We define $H_m^+,H_m^-:\mathbb R\times\mathbb R^n\times\mathcal{M}(n)\to \mathbb R$ through
\begin{eqnarray}\label{par6}
H_m^+(\xi,p,M)&=&\tilde H_m^+(p,M)+r\xi,\notag\\
H_m^-(\xi,p,M)&=&\tilde H_m^-(p,M)+r\xi.
\end{eqnarray}
Next, we introduce terminal value problems involving Bellman-Isaacs type equations associated to the game with bounded controls.
\begin{eqnarray}\label{par9ap-}
\partial_t u-H_m^+(u,Du,D^2 u)&=&0\ \ \ \ \ \ \ \mbox{ in }\mathbb R^n\times (0,T),\notag\\
u(x,T)&=&g(x)\ \ \ \mbox{ on }\mathbb R^n,
\end{eqnarray}
\begin{eqnarray}\label{par10ap-}
\partial_t u-H_{m}^-(u,Du,D^2 u)&=&0\ \ \ \ \ \ \ \mbox{ in }\mathbb R^n\times (0,T),\notag\\
u(x,T)&=&g(x)\ \ \ \mbox{ on }\mathbb R^n.
\end{eqnarray}
A suitable concept of solution to the above equations is the notion of viscosity solutions. Recall our standing assumption \eqref{eq:standing} for $g$.

\begin{definition}
\label{vis1}
\begin{enumerate}[(a)]
\item A function $\bar u^+_m\in \mbox{LSC}_l (\mathbb R^n\times [0,T])$ is a viscosity supersolution to \eqref{par9ap-} if $\bar u^+_m(x,T)\geq g(x)$ for all $x\in\mathbb R^n$ and if the following holds.  If $(x_0,t_0)\in\mathbb R^n\times (0,T)$ and we have $\phi\in C^{1,2}(\mathbb R^n\times [0,T])$ such that
\begin{eqnarray*}
(i)&&  \bar u^+_m(x_0, t_0) = \phi(x_0, t_0),\\
(ii)&& \bar u^+_m(x, t) > \phi(x, t)\mbox{ for }(x,t)\neq (x_0,t_0),
\end{eqnarray*}
then
\begin{eqnarray*}
\partial_t \phi(x_0,t_0)\le H_m^+(\bar u^+_m(x_0,t_0),D\phi (x_0,t_0),D^2 \phi(x_0, t_0)).
\end{eqnarray*}
\item A function $\underbar{u}^+_m\in \mbox{USC}_l (\mathbb R^n\times [0,T])$ is a viscosity subsolution to  \eqref{par9ap-} if $\underbar{u}^+_m(x,T)\leq g(x)$ for all $x\in\mathbb R^n$ and if the following holds.  If $(x_0,t_0)\in\mathbb R^n\times (0,T)$ and we have $\phi\in C^{1,2}(\mathbb R^n\times [0,T])$ such that
\begin{eqnarray*}
(i)&&  \underbar{u}^+_m(x_0, t_0) = \phi(x_0, t_0),\\
(ii)&& \underbar{u}^+_m(x, t) < \phi(x, t)\mbox{ for } (x,t)\neq (x_0,t_0),
\end{eqnarray*}
then
\begin{eqnarray*}
\partial_t \phi(x_0,t_0)\ge H_m^+(\underbar{u}^+_m(x_0,t_0),D\phi (x_0,t_0),D^2 \phi(x_0, t_0)).
\end{eqnarray*}
\item If $u_m$ is both a  viscosity supersolution and a viscosity subsolution to
\eqref{par9ap-},  then $u_m$ is a viscosity solution to \eqref{par9ap-}.
\item The definitions for the  equation \eqref{par10ap-} are analogous with $H_m^+$, $\bar u^+_m$, $\underbar{u}^+_m$ replaced by $H_m^-$, $\bar u_m^-$, $\underbar{u}_m^-$.
\end{enumerate}
\end{definition}
\begin{remark}
\label{rem:deg-ell}
Note that $H_m^+(u,p,X)$ is continuous with respect to  $u,p,X$ also when $p=0$. In addition, $H_m^+$ is degenerate elliptic in the sense that
\begin{equation}
\label{eq:degenerate-elliptic}
\begin{split}
H_m^+(u,p,X)\le H_m^+(u,p,Y),
\end{split}
\end{equation}
whenever $X\ge Y$. The analogous statements hold for $H_m^-(u,p,X)$.
\end{remark}

\subsection{Bellman-Isaacs equation with bounded action sets: existence and uniqueness of viscosity solutions}

\begin{lemma} \label{thm:value-} \label{lem:comparison-m} Let $\underbar{u}_m^+, \bar u_m^+\in \mbox{C}_l(\mathbb R^n\times [0,T])$ and $\underbar{u}_m^-, \bar{u}_m^-\in \mbox{C}_l(\mathbb R^n\times [0,T])$, be viscosity sub- and supersolutions to \eqref{par9ap-} and \eqref{par10ap-}, respectively.
Then,
$$
\underbar{u}_m^-\leq \bar{u}_m^- \quad \text{and}\quad   \underbar{u}_m^+\leq \bar{u}_m^+.
$$
\end{lemma}
For the proof of the above comparison principle,  see \cite{gigagis91}, in particular see the argument starting from page 27 in \cite{gigagis91}. Similarly,
by comparison with a sufficiently large constant it follows that solutions are not merely of linear growth: they are bounded.

\begin{lemma} \label{thm:value-ha} Let $y\in\mathbb R^n$ and let $L$ be the Lipschitz constant of $g$. Consider $0<\ep\ll 1$, and let
\begin{eqnarray*}
\bar{w}(x,t)=g(y)+\frac A{\ep^2}(T-t)+2L(|x-y|^2+\ep)^{1/2},\notag\\
\underline{w}(x,t)=g(y)-\frac A{\ep^2}(T-t)-2L(|x-y|^2+\ep)^{1/2}.
\end{eqnarray*}
Then we can choose $A$, independent of $y$, $\ep$ and $m$, so that
$\bar{w}$ and $\underline{w}$ are viscosity super- and subsolutions to \eqref{par9ap-} and  \eqref{par10ap-}.
\end{lemma}
\begin{proof} We only prove the result for \eqref{par9ap-} as the proof for \eqref{par10ap-} is analogous. First, we immediately see that
\[
\begin{split}
\underline{w}(x,T)\leq g(x)\leq \bar{w}(x,T)
,\end{split}
\]
whenever $x\in \mathbb R^n$. To prove that $\bar{w}$ is a  viscosity supersolution to \eqref{par9ap-} we need to verify that
 \[
\begin{split}
\partial_t \bar{w}(x,t)-&H_m^+(\bar{w}(x,t),D\bar{w}(x,t),D^2 \bar{w}(x,t))\leq 0
,\end{split}
\]
whenever $(x,t)\in\mathbb R^n\times[0,T]$.  In the following, we split
\[
\begin{split}
\Phi:=\Phi(\theta^+,\theta^-,d^+,d^-,D\bar{w},D^2\bar{w})
\end{split}
\]
into
\begin{eqnarray}\label{par2ll}
\Phi&=&\Phi_1+\Phi_2,
\end{eqnarray}
where
\[
\begin{split}
\Phi_1&:=-\half (\theta^+-\theta^-)'\Sigma D^2\bar{w}\Sigma (\theta^+-\theta^-)\notag\\
&\hspace{1 em}-\half \tr(\Sigma^2 D^2\bar{w})-
\mu\cdot D\bar{w},\notag\\
\Phi_2&:=-(d^++d^-)(\theta^++\theta^-)\cdot D\bar{w}.
\end{split}
\]
By a straightforward calculation, we see that
\begin{eqnarray}\label{par2ll+}
&&|\Phi_1|\leq cL(|x-y|^2+\ep)^{-1/2}, 
\end{eqnarray}
 for all $(\theta^-,d^-)\in\mathcal{H}_m$, $(\theta^+,d^+)\in\mathcal{H}_m$, and for some $c$ independent of $y$, $L$, $m$ and $\ep$. In particular, there is no $m$ dependence as $d^\pm$ plays no explicit role in $\Phi_1$.

We next estimate
\[
\begin{split}
&\sup_{(\theta^-,d^-)\in\mathcal{H}_m}\ \ \inf_{(\theta^+,d^+)\in\mathcal{H}_m}\Phi_2\notag\\
&=2L\sup_{(\theta^-,d^-)\in\mathcal{H}_m}\ \ \inf_{(\theta^+,d^+)\in\mathcal{H}_m} \Big(
-(d^++d^-)(\theta^++\theta^-)\cdot \frac {(x-y)}{(|x-y|^2+\ep)^{1/2}}\Big).
\end{split}
\]
We can, without loss of generality, assume that  $x\neq y$. Let $\theta^-=-(x-y)/|x-y|$ and note that
\[
\begin{split}
-(d^++d^-)(\theta^++\theta^-)\cdot \frac {(x-y)}{(|x-y|^2+\ep)^{1/2}}\geq 0.
\end{split}
\]
Combining this estimate and \eqref{par2ll+}
with \eqref{par2ll}, recalling the term$-r \ol w(x,t)$, we see that
\[
\begin{split}
\bar{w}_t(x,t)-&H_m^+(\bar{w}(x,t),D\bar{w}(x,t),D^2 \bar{w}(x,t))\notag\\
&\leq-\frac A{\ep^2}+ cL(|x-y|^2+\ep)^{-1/2}\notag\\
&\leq -\frac A{\ep^2}+ cL\ep^{-1/2},
\end{split}
\]
whenever $(x,t)\in\mathbb R^n\times[0,T]$. Hence, if we let $A=2cL\ge cL\eps^{3/2}$ then  $-A{\ep^{-2}}+ cL\ep^{-1/2}\leq 0$ and we can conclude that
$\bar{w}$ is a supersolution to \eqref{par9ap-}.  The proof that $\underline{w}$ is a  viscosity subsolution to \eqref{par9ap-} is almost  analogous. In particular, we first note that
\[
\begin{split}
-r\underline{w}(x,t)=-r\big(g(y)-\frac A{\ep^2}(T-t)-2L(|x-y|^2+\ep)^{1/2}\big)\ge -r g(y).
\end{split}
\]
Then, observing that the right hand side is bounded, adjusting the constants, and repeating the argument above we can conclude that $\underline w$ is a subsolution.
\end{proof}
From now on we fix $A$ so that
$\bar{w}$ and $\underline{w}$ are viscosity super- and subsolutions, as stated in
Lemma \ref{thm:value-ha}, to \eqref{par9ap-} and \eqref{par10ap-}.

\begin{lemma}
 \label{cor:value-ha}
If $u_m^+$ and $u_m^-$ are viscosity solutions to
 \eqref{par9ap-} and \eqref{par10ap-}, respectively, then
 \[
\begin{split}
\underline{w}&\leq {u}_m^\pm\leq \bar{w}.
\end{split}
\]
\end{lemma}
\begin{proof} The lemma is an immediate consequence of Lemma \ref{thm:value-ha} and  Lemma \ref{thm:value-}.
\end{proof}

\begin{lemma} \label{thm:value-haha}
 There exist unique viscosity solutions $u_m^+$ and $u_m^-$  to
\eqref{par9ap-} and \eqref{par10ap-}, respectively.
\end{lemma}
The proof of this result can be found in \cite{giga06} and it is based on Perron's method. Using this method, one constructs the so called upper and lower Perron solutions by taking $\inf/\sup$ over the suitable super/subsolutions. That the constructed solutions assume the correct terminal data can then be proved by using the above barriers.

\setcounter{equation}{0} \setcounter{theorem}{0}

\section{Solving the stochastic dynamic game\\ with bounded action sets}

  The purpose of the section is to prove the following lemma.

\begin{lemma} \label{lem:value} Let $u_m^+$ and $u_m^-$ be the unique solutions to \eqref{par9ap-} and \eqref{par10ap-}, respectively, ensured
by Lemma \ref{thm:value-haha}. Then, \begin{eqnarray*}
&&u_m^+(x,t)=U_m^+(x,t):=\sup_{\rho^+\in{\str}_m}\inf_{A^-\in {\ac}_m}J^{(x,t)}(\rho^+(A^-),A^-),\notag\\
&&u_m^-(x,t)=U_m^-(x,t):=\inf_{\rho^-\in{\str}_m}\sup_{A^+\in {\ac}_m}J^{(x,t)}(A^+,\rho^-(A^+)),
\end{eqnarray*}
whenever  $(x,t)\in \mathbb R^n\times [0,T]$.
\end{lemma}

\subsection{Proof of Lemma \ref{lem:value} assuming additional regularity on $u_m^\pm$} We here prove Lemma \ref{lem:value} assuming smoothness on  $u_m^+$ and $u_m^-$.
\begin{lemma}\label{lem:value-for-smooth}
 Let $u_m^+$ and $u_m^-$ be the unique solutions to \eqref{par9ap-} and \eqref{par10ap-}, respectively, ensured
by Lemma \ref{thm:value-haha}.  Assume, in addition, that
$u_m^\pm,\partial_tu_m^\pm, Du_m^\pm,D^2u_m^\pm$, are Lipschitz continuous in $\mathbb R^n\times [0,T)$. Then Lemma \ref{lem:value} holds.
\end{lemma}
\begin{proof} The proof is based on the connection between solutions and value functions provided by the Ito formula, in connection with suitable discretized controls chosen based on the solution.  The discretization error can be estimated by utilizing the smoothness assumptions. At the end, we pass to a limit with the discretization parameter. We only supply the proof
in the case of $u_m^-$, the proof for $u_m^+$ being analogous.

Given $k\in \{1,2,\ldots\}$ and $(x,t)\in \Rn\times [0,T]$, we can choose, as $u_m^-$ is a solution to \eqref{par10ap-},
a control $(\theta_0^+,d_0^+)\in \mathcal{H}_m$ such that
\begin{equation}
\label{eq:small}
\begin{split}
&\hspace{-1 em}\sup_{(\theta^-,d^-)\in \mathcal{H}_m}\biggl\{ \Phi(\theta_0^+,\theta^-,d_0^+,d^-,Du_m^-(x,t),D^2u_m^-(x,t))+ru_m^-(x,t)\biggr \}\\
&\hspace{3 em}\leq \partial_t u_m^-(x,t)+ k^{-1}.
\end{split}
\end{equation}
Based on $(\theta_0^+,d_0^+)=(\theta_{0,1}^+,....,\theta_{0,n}^+, d_0^+)$ and an arbitrary, but fixed, control $(\theta^-,d^-)\in {\ac}_m$, we let
$X^0(s):=X^{0,(x,t)}(s)$ be defined as in \eqref{eq:dynamics} assuming also the initial condition $X^0(t)=x$. In the following, we let
\[
\begin{split}
\Phi_0^{X}(s):&=\Phi(\theta_0^+,\theta^-(s),d_0^+,d^-(s),Du_m^-(X^0(s),s),D^2u_m^-(X^0(s),s)),\\
\Phi_0^x(s):&=\Phi(\theta_0^+,\theta^-(s),d_0^+,d^-(s),Du_m^-(x,s),D^2u_m^-(x,s)).
\end{split}
\]
As $u_m^-,\partial_tu_m^-, Du_m^-,D^2u_m^-$ are Lipschitz continuous in $\mathbb R^n\times [0,T)$, we can apply the Ito formula to $u_m^-(X^0(s),s)$ and
\[
\begin{split}
du_m^-&(X^0(s),s)=\partial_t u_m^-(X^0(s),s)ds+\sum_{i=1}^n \partial_{x_i}u_{m}^-(X^0(s),s)dX^0_i(s)\\
&+\half \sum_{i,j=1}^n  \partial_{x_ix_j}u_{m}^-(X^0(s),s) \, dX^0_i(s)\, dX^0_j(s)\\
=&(\partial_t u_m^-( X^0(s),s)- \Phi_0^{X}(s))ds\\
&+\sum_{i=1}^n \partial_{x_i}u_{m}^-( X^0(s),s)(\sigma_idW_i(s)+ \sigma_i(\theta^+_{0,i}-\theta_i^-(s))dW_{n+1}(s)).
\end{split}
\]
Using this
\[
\begin{split}
d(&e^{-rs}u_m^-( X^0(s),s))\\
=&e^{-rs}(\partial_t u_m^-( X^0(s),s)- \Phi_0^{X}(s)-ru_m^-(X^0(s),s))ds\\
&+e^{-rs}\sum_{i=1}^n \partial_{x_i}u_{m}^-( X^0(s),s)(\sigma_idW_i(s)+ \sigma_i(\theta^+_{0,i}-\theta_i^-(s))dW_{n+1}(s)).
\end{split}
\]
Hence, if we let $\Delta t=(T-t)/k$, then
\[
\begin{split}
\mathbb E& [e^{-r(t+\Delta t)}u_m^-( X^0(t+\Delta t),t+\Delta t))-e^{-rt}u_m^-( X^0(t),t)]\\
&=\mathbb E \biggl [\int_t^{t+\Delta t}e^{-rs}(\partial_t u_m^-( X^0(s),s)- \Phi_0^{X}(s)-ru_m^-( X^0(s),s))ds\biggr ],
\end{split}
\]
and
\begin{eqnarray}
\label{eq:small++++}
&&u_m^-(x,t)\notag\\
&=&\mathbb E [e^{-r\Delta t}u_m^-( X^0(t+\Delta t),t+\Delta t)]\\
&-&\mathbb E \biggl [\int_t^{t+\Delta t}e^{-r(s-t)}(\partial_t u_m^-( X^0(s),s)- \Phi_0^X(s)-ru_m^-( X^0(s),s))ds\biggr ].\notag
\end{eqnarray}
We let
\begin{equation}
\label{eq:small++++l}
\begin{split}
I_1&=-\mathbb E \biggl [\int_t^{t+\Delta t}e^{-r(s-t)}(\partial_t u_m^-( X^0(s),s)-\partial_t u_m^-(x,t))ds\biggr ],\\
I_2&=-\mathbb E \biggl [\int_t^{t+\Delta t}e^{-r(s-t)}( \Phi_0^x(t)- \Phi_0^X(s))ds\biggr ],\\
I_3&=-\mathbb E \biggl [\int_t^{t+\Delta t}e^{-r(s-t)}r(u_m^-( X^0(t),t))-u_m^-( X^0(s),s))ds\biggr ].
\end{split}
\end{equation}
Using this notation, we observe that
\begin{equation}
\label{eq:small++++k}
\begin{split}
&\hspace{0 em}u_m^-(x,t)\\
&\hspace{0 em}=\mathbb E [e^{-r\Delta t}u_m^-( X^0(t+\Delta t),t+\Delta t)]+I_1+I_2+I_3\\
&\hspace{1 em}-\mathbb E \biggl [\int_t^{t+\Delta t}e^{-r(s-t)}\Big(\partial_t u_m^-(x,t)- \Phi_0^x(t)-ru_m^-(x,t)\Big) ds\biggr ].
\end{split}
\end{equation}
Next, using  \eqref{eq:small} we see that
\begin{equation}
\label{eq:use-of-eq}
\begin{split}
-\mathbb E& \biggl [\int_t^{t+\Delta t}e^{-r(s-t)}\Big(\partial_t u_m^-(x,t)- \Phi_0^x(t)-ru_m^-(x,t)\Big)ds\biggr ]\\
&\leq k^{-1}\Delta t.
\end{split}
\end{equation}
Furthermore, by Lipschitz continuity  of $u_m^-, \partial_tu_m^-, Du_m^-,D^2u_m^-$, we can conclude that
\begin{equation}
\label{eq:discretization-error}
\begin{split}
|I_1|+|I_2|+|I_3|&\leq c \mathbb E \biggl [\int_t^{t+\Delta t}(| X^0(s)-x|+\Delta t)ds\biggr ]\\
&\leq c(\Delta t)^2+ c\mathbb E \biggl [\int_t^{t+\Delta t}| X^0(s)-x|ds\biggr ],
\end{split}
\end{equation}
for some generic constant $c$.  To estimate the expectation in the previous estimate, we will have to use the equation satisfied by
$ X^0(s)$. Recall that
\[
\begin{split}
 X^0_i(s)-x_i=&\int_t^{s}\Big(\mu_i+(d^+_{1}+d^-(\tau))(\theta_{0,i}^++\theta^-_i(\tau))\Big)d\tau\\
&+\int_t^{s}\sigma_idW_i(\tau)+ \int_t^{s}\sigma_i (\theta_{0,i}^+-\theta^-_i(\tau))dW_{n+1}(\tau).
\end{split}
\]
Hence, simply using the H{\"o}lder inequality and the Ito isometry, we see that
\[
\begin{split}
\mathbb E \biggl [\int_t^{t+\Delta t}| X^0(s)-x|ds\biggr ]&= \biggl [\int_t^{t+\Delta t}\mathbb E[| X^0(s)-x|]ds\biggr ]\\
&\leq \biggl [\int_t^{t+\Delta t}(\mathbb E[| X^0(s)-x|^2])^{1/2}ds\biggr ]\notag\\
&\leq c(\Delta t)^{3/2},
\end{split}
\]
and where $c$ is allowed to depend on $m$ defining the class ${\ac}_m$. Combining
the above estimates, we conclude that
\[
\begin{split}
|I_1|+|I_2|+|I_3|\leq c((\Delta t)^2+ (\Delta t)^{3/2})\leq c(\Delta t)^{3/2}.
\end{split}
\]
Hence returning to \eqref{eq:small++++}, also recalling \eqref{eq:small++++l} and \eqref{eq:small++++k}, we see that
\begin{eqnarray}\label{eq:small++++ag}
u_m^-(x,t)&\leq&\mathbb E [e^{-r\Delta t}u_m^-( X^{0,(x,t)}(t+\Delta t),t+\Delta t)] \notag\\
&&+c(\Delta t)^{3/2} +k^{-1}\Delta t,
\end{eqnarray}
where $\Delta t=(T-t)/k$.

We now use
\eqref{eq:small++++ag}  in an iterative construction. Indeed, we let  $t_j=t+j\Delta t$ for $j=0,...,k-1$ and we first note, using \eqref{eq:small++++ag},
that for $j=0$ we have
\begin{eqnarray}
u_m^-(x,t_0)&\leq&\mathbb E [e^{-r\Delta t}u_m^-( X^{0,(x,t_0)}(t_1),t_1)] \notag\\
&&+c(\Delta t)^{3/2} +k^{-1}\Delta t.\notag
\end{eqnarray}
We next consider $j=1$. Then, using that $u_m^-$ is a solution to \eqref{par10ap-}, that  $\mathcal{H}_m$ is a separable metric space, and uniform
continuity, it follows that  there exist a sequence $\{(\theta_{1l}^+, d_{1l}^+)\}_{l=1}^\infty \subset \mathcal{H}_m$ and a covering $\{B(y_{1l},r_{1l})\}_{l=1}^\infty$ of $\Rn$, such that
\[
\begin{split}
\sup_{(\theta^-,d^-)\in \mathcal{H}_m}&\biggl \{\Phi( \theta_{1l}^+,\theta^-, d_{1l}^+,d^-,Du_m^-(y,t_1),D^2u_m^-(y,t_1))+ru_m^-(y,t_1)\biggr \}\\
&\leq \partial_t u_m^-(y,t_1)+ k^{-1}\qquad \mbox{whenever}\quad y\in B(y_{1l},r_{1l}).
\end{split}
\]
We let $\psi_1:\mathbb R^n\to \mathcal{H}_m$
be defined as
$$\psi_1(y)=(\psi_1^\theta(y),\psi_1^d(y)):=( \theta_{1l}^+, d_{1l}^+)\mbox{ if } y\in B(y_{1l},r_{1l})\setminus\cup_{i=1}^{l-1}B(y_{1i},r_{1i}).$$
Furthermore, this time we let
\[
\begin{split}
 \Phi_1^y(s):=\Phi(\psi_1^\theta(y),\theta^-,\psi_1^d(y),d^-,Du_m^-(y,t_1),D^2u_m^-(y,t_1)).
\end{split}
\]
Then,
\begin{eqnarray}
\sup_{(\theta^-,d^-)\in \mathcal{H}_m}\biggl \{ \Phi_1^y(s)+ru_m^-(y,t_1)\biggr \}\leq& \partial_t u_m^-(y,t_1)+ k^{-1},\notag
\end{eqnarray}
whenever $y\in \Rn$.
We now let
\[
\begin{split}
( \theta_1^+(s), d_1^+(s))=( \theta_0^+,  d_0^+),
\end{split}
\]
for $s\in[t_0,t_1)=[t,t+\Delta t)$,
and
\begin{equation}
\label{eq:smallha2}
\begin{split}
( \theta_1^+(s), d_1^+(s))=(\psi_1^\theta( X^{0}(t_1)),\psi_1^d(( X^{0}(t_1))),
\end{split}
\end{equation}
for $s\in[t_1,t_2)=[t+\Delta t,t+2\Delta t)$. In this way, we have constructed a new control $( \theta_1^+, d_1^+)\in {\ac}_m$. Next, with this $( \theta_1^+, d_1^+)\in {\ac}_m$ and an arbitrary, but fixed, control $(\theta^-,d^-)\in {\ac}_m$, we construct
$ X^1(s)=X^{1,(x,t)}(s)$ for $s\in [t_0,t_2)$ satisfying the initial condition  $X^1(t)=x$  and the dynamics in \eqref{eq:dynamics}. By construction, it follows that $X^{1}(s)=X^{0}(s)$ for $s\in [t_0,t_1)$. We can now repeat the argument above to conclude that
\begin{eqnarray}\label{eq:small++++agpp}
u_m^-( X^0(t_1),t_1)&\leq&\mathbb E [e^{-r\Delta t}u_m^-( X^{1}(t_2),t_2)|X^0(t_1)]\notag\\
&& +c(\Delta t)^{3/2} +k^{-1}\Delta t.\notag
\end{eqnarray}
In particular, we see that
\begin{eqnarray}
\label{eq:small++++agl}
\mathbb E[e^{-r\Delta t}u_m^-( X^{0}(t_1),t_1)]&\leq&\mathbb E [e^{-2r\Delta t}u_m^-( X^{1}(t_2),t_2)]\\
&&+c(\Delta t)^{3/2} +k^{-1}\Delta t.\notag
\end{eqnarray}
Combining this with \eqref{eq:small++++ag}, we conclude that
%
%
\begin{eqnarray}
u_m^-(x,t)&\leq&\mathbb E [e^{-2r\Delta t}u_m^-( X^{1}(t_2),t_2)]\notag\\
&& +c2(\Delta t)^{3/2} +2k^{-1}\Delta t.\notag
\end{eqnarray}
By carefully iterating the above argument, we get a sequence of controls $( \theta_j^+, d_j^+)\in {\ac}_m$, for $j\in\{0,1,...,k-1\}$, and a sequence of processes
$X^j(s)=X^{j,(x,t)}(s)$ based on  the controls $( \theta_j^+, d_j^+)$ and an arbitrary, but fixed, $(\theta^-,d^-)\in {\ac}_m$. In particular, we have
\begin{eqnarray}
u_m^-(x,t)&\leq&\mathbb E [e^{-jr\Delta t}u_m^-( X^{j}(t_j),t_j)]\notag\\
&& +cj(\Delta t)^{3/2} +jk^{-1}\Delta t,\notag
\end{eqnarray}
for all $j\in\{1,...,k\}$. Applying this inequality with $j=k$, we conclude  that
\begin{eqnarray}\label{eq:small++++aglit+}
u_m^-(x,t)&\leq&\mathbb E [e^{-r(T-t)}u_m^-( X^{k}(T),T))]+c(T-t)(\Delta t)^{1/2} +\Delta t\notag\\
&=&\mathbb E [e^{-r(T-t)}g( X^{k}(T))]+c(T-t)(\Delta t)^{1/2} +\Delta t,
\end{eqnarray}
where we  also used $u_m^-(x,T)=g(x)$ in the last line.

Summing up, given $k\in \{1,2,\ldots\}$ and $(x,t)\in \Rn\times [0,T]$, we have constructed controls $( \theta_k^+, d_k^+)$ such that for an arbitrary, but fixed, $(\theta^-,d^-)\in {\ac}_m$,  \eqref{eq:small++++aglit+} holds with $X^{k}$ defined as in \eqref{eq:dynamics} based on
 $( \theta_k^+, d_k^+)$, $(\theta^-,d^-)$, and $X^{k}(t)=x$.

 Next, consider $\rho^-\in\mathcal{ S}_m$. Based on the above argument, we now construct
 $( \theta_j^+, d_j^+)$ and $( \theta^-, d^-)|_{[t_j,t_{j+1})}$, $j=0,1,...,k-1$,  using $\rho^-$. Indeed, given  $(x,t)\in \Rn\times [0,T]$, we first let
 $( \theta_0^+, d_0^+)$ be as above and set
\begin{eqnarray}
(\theta^-,d^-)|_{[t_0,t_1)}:=\rho^-(\theta^+_0, d^+_0)|_{[t_0,t_1)}.\notag
\end{eqnarray}
Then, having defined $(\theta^-,d^-)|_{[t_0,t_1)}$, we may define $( \theta^+_1, d^+_1)$ on $[t_0,t_2)$ as in \eqref{eq:smallha2}. Repeating the above argument, we set
\begin{eqnarray}
(\theta^-,d^-)|_{[t_1,t_2)}:=\rho^-(\theta^+_1, d^+_1)|_{[t_1,t_2)}.\notag
\end{eqnarray}
In particular, proceeding inductively, we can based on   $\rho^-\in\mathcal{ S}_m$ construct the controls $( \theta_k^+, d_k^+)$ and
$( \theta^-, d^-)$ such that
\begin{eqnarray}
u_m^-(x,t)&\leq&\mathbb E [e^{-r(T-t)}g( X^{k}(T))]+c(T-t)(\Delta t)^{1/2} +\Delta t\notag\\
&\leq&\sup_{A^+\in {\ac}_m}J^{(x,t)}(A^+,\rho^-(A^+))\notag\\
&&+c(T-t)(\Delta t)^{1/2} +\Delta t,\notag
\end{eqnarray}
for all $k$, $\Delta t=(T-t)/k$, and for all $\rho^-\in\mathcal{ S}_m$. Letting $k\to \infty$, we end up with
\begin{eqnarray}\label{eq:small++++aglit+apa+}
u_m^-(x,t)\le \inf_{\rho^-\in\mathcal{ S}_m}\sup_{A^+\in {\ac}_m}J^{(x,t)}(A^+,\rho^-(A^+)).
\end{eqnarray}

To prove the opposite inequality, we again note, as $u_m^-$ is a solution to \eqref{par10ap-} and $\Phi$ is uniformly continuous, that we can choose, given $k\in \{1,2,\ldots \}$ and $(x,t)\in \Rn\times [0,T]$, a countable set  of  controls $\{\theta^-_{1i},d^-_{1i}\}_{i}\subset \mathcal{H}_m$ and a covering 
\[
\begin{split}
\{ B_{r_l}(x_l) \times B_{r'_l}(\theta^+_l,d_l^+)\}_l,
\end{split}
\]
such that
\begin{equation}
\begin{split}
&
\biggl\{ \Phi(\theta^+,\theta^-_{1l},d^+,d^-_{1l},Du_m^-(x,t),D^2u_m^-(x,t))+ru_m^-(x,t)\biggr \}\\
&\hspace{3 em}\geq \partial_t u_m^-(x,t)- k^{-1},  \notag
\end{split}
\end{equation}
if $(x,(\theta^+,d^+))\in  B_{r_l}(x_l) \times B_{r'_l}(\theta^+_l,d_l^+)$. We define a map
\[
\begin{split}
\psi_0(x,(\theta^+,d^+)):=(\theta^-_{1l},d^-_{1l}),
\end{split}
\]
if $(x,(\theta^+,d^+)) \in B_{r_l}(x_l) \times B_{r'_l}(\theta^+_l,d_l^+)\setminus \cup_{i=1}^{l-1} B_{r_i}(x_i) \times B_{r'_i}(\theta^+_i,d_i^+)$.
 Then, by arguing as above, we deduce that we can construct controls  such that
\begin{eqnarray}\label{eq:small++++aglit+ml}
u_m^-(x,t)&\geq&\mathbb E [e^{-r(T-t)}u_m^-( X^{k}(T),T))]-c(T-t)(\Delta t)^{1/2} -\Delta t\notag\\
&=&\mathbb E [e^{-r(T-t)}g( X^{k}(T))]-c(T-t)(\Delta t)^{1/2} -\Delta t,
\end{eqnarray}
for an arbitrary, but fixed $(\theta^+,d^+)\in {\ac}_m$. Here $X^{k}$ is defined as in \eqref{eq:dynamics}, based on the controls above, and $X^{k}(t)=x$.

We can now use $(\theta^-_{il},d^-_{il}),\ i=1,2,\ldots,k$, constructed  above to define our strategy given $(\theta^+,d^+)$.
In particular, using \eqref{eq:small++++aglit+ml}, we see that
\begin{eqnarray}
u_m^-(x,t)&\geq&J^{(x,t)}(A^+,\rho_k^-(A^+))-c(T-t)(\Delta t)^{1/2} -\Delta t,\notag
\end{eqnarray}
and, hence
\begin{eqnarray}
u_m^-(x,t)&\geq& \inf_{\rho^-\in\mathcal{ S}_m}J^{(x,t)}(A^+,\rho^-(A^+))\notag\\
&&-c(T-t)(\Delta t)^{1/2} -\Delta t,\notag
\end{eqnarray}
for all $k$, $\Delta t=(T-t)/k$, and for all $A^+\in {\ac}_m$. In particular, letting $k\to \infty$, we can conclude that
\begin{eqnarray}
u_m^-(x,t)\ge \inf_{\rho^-\in\mathcal{ S}_m}\sup_{A^+\in {\ac}_m}J^{(x,t)}(A^+,\rho^-(A^+)).\notag
\end{eqnarray}
Combining this with \eqref{eq:small++++aglit+apa+}, we see that the proof of Lemma \ref{lem:value-for-smooth} for $u_m^-$ is complete. \end{proof}

\subsection{Proof of Lemma \ref{lem:value}} In the following, we only supply the proof
in the case of $u_m^-$ in $\Rn\times[0,T]$, the proof for $u_m^+$ being analogous. Given a large non-negative integer $j$, we let $T_j=T-\frac 1j$ and  $\Rn_j:=\Rn\times[\frac1j, T_j]$. Given $j$ fixed we introduce, for $\epsilon>0$ small, the $\sup$-convolution
\[
\begin{split}
u_\eps(x_0,t_0):=\sup_{(x,t)\in \R^n\times [0,T]}\{u_m^-(x,t)-\frac{(t_0-t)^2+(x_0-x)^2}{2\eps}\},
\end{split}
\]
  whenever $(x_0,t_0)\in \R^n_j$. Then $u_m^-(x_0,t_0)\leq u_\eps(x_0,t_0)$ whenever $(x_0,t_0)\in \R^n_j$ and $u_\eps$ is a semi-convex function, i.e. there exists a constant $c>0$ such that $u_{\eps}(x,t)+c(\abs{x}^2+t^2)$ is convex. Furthermore, provided $\eps$ is small enough,
   \[
\begin{split}
H^-_m(x,t,D u_\eps,D^2 u_\eps)\le \partial_t u_\eps(x,t)+ \om({\eps}),
\end{split}
\]
for a.e.\ $(x,t)\in \R^n_j$ and where $\om({\eps})$ is a  bounded modulus which depends on the continuity of $u_m^-$.  For details and properties of $\sup$-convolutions, see for example \cite{crandallil92}, \cite{ishii95} and \cite{lindqvist12}.

 Next, using the standing assumptions on $g$, see \eqref{eq:standing}, and the comparison principle, we see that $0\leq u_m^-(x,t)\leq L$ for all
 $(x,t)\in \mathbb R^n\times [0,T]$. Furthermore, using the boundedness of $u_m^-$, it follows that the supremum used in the definition of  $u_\eps(x_0,t_0)$ is obtained at some point $(x^*,t^*)$. In particular,
  \[
\begin{split}
0&\le u_m^-(x_0,t_0)\le u_\eps(x_0,t_0)=u_m^-(x^*,t^*)-\frac{(t_0-t^*)^2+(x_0-x^*)^2}{2\eps}\\
&\le L-\frac{(t_0-t^*)^2+(x_0-x^*)^2}{2\eps},
\end{split}
\]
and hence
\[
\begin{split}
 \sqrt{(t_0-t^*)^2+(x_0-x^*)^2}\le \sqrt{2L{\eps}},
\end{split}
\]
where we deduce a condition $\sqrt{2L\eps}<1/j$ for $\eps$.

Given $\delta>0$ small, we let $\eta_{\delta}$ denote a standard mollifier in $\mathbb R^{n+1}$. Restricting $\delta\ll ((j-1)^{-1}-j^{-1})/2$, we see that $u_\eps^{\delta}(x,t):=u_\eps*\eta_{\delta}(x,t)$, the convolution of $u_\eps$ and $\eta_{\delta}$, is well-defined whenever $(x,t)\in\mathbb R^n_{j-1}$. Then,
\[
\begin{split}
u_\eps^\delta&\to u_\eps,\quad \text{uniformly on }\R^n_{j-1},\\
D u_\eps^\delta&\to D u_\eps,\quad \text{a.e.\ in }\R^n_{j-1},\\
\partial_tu_\eps^\delta&\to \partial_tu_\eps,\quad \text{a.e.\ in }\R^n_{j-1},\\
D^2 u_\eps^\delta&\to D^2 u_\eps,\quad \text{a.e.\ in }\R^n_{j-1}.
\end{split}
\]
The last statement is based on Alexandrov's theorem, see for example  Section 6 of \cite{evansg92} or \cite{juutinenj12}.
Furthermore,
\[
\begin{split}
H^-_m(x,t,D u_\eps^\delta,D^2 u_\eps^\delta)\le \partial_t u_\eps^\delta (x,t)+\om(\eps)+\gamma_\delta(x,t),
\end{split}
\]
 a.e.\ in $\R^n_{j-1}$ where $\gamma_{\delta}(x,t)\to 0$ as $\delta\to 0$,
and $\gamma_{\delta}$ is uniformly continuous (this is a result for the standard convolution; the modulus of continuity is not claimed to be uniform in $\delta$) and bounded uniformly in $\delta$ (recall uniform semiconvexity  and uniform Lipschitz continuity of $u_{\eps}^\delta$ with respect to $\delta$).
Now, using that $u_\eps^\delta,\partial_tu_\eps^\delta, Du_\eps^\delta,D^2u_\eps^\delta$ are Lipschitz continuous in $\mathbb R^n_{j-1}$, we can
argue as in the proof of Lemma \ref{lem:value-for-smooth} to conclude that
\begin{equation}
\label{eq:key-mollif}
\begin{split}
u_\eps^\delta(x,t)\le & \inf_{\rho^-\in{\str}_m}\sup_{A^+\in {\ac}_m} \mathbb E\biggl[\int_t^{T_{j-1}} e^{-r(s-t)} h_\eps^\delta(X(s),s)\ud s\\
&\hspace{6 em}+e^{-r(T_{j-1}-t)}u_\eps^\delta( X(T_{j-1}),T_{j-1})\biggr],
\end{split}
\end{equation}
whenever $(x,t)\in \mathbb R^n_{j-1}$ and where $ h_\eps^\delta:=\om(\eps)+\gamma_\delta$. Using this,  \eqref{eq:small} becomes
\[
\begin{split}
&\hspace{-1 em}\sup_{(\theta^-,d^-)\in \mathcal{H}_m}\biggl\{ \Phi(\theta_0^+,\theta^-,d_0^+,d^-,Du_\eps^\delta(x,t),D^2u_\eps^\delta(x,t))+ru_\eps^\delta(x,t)\biggr \}\\
&\hspace{3 em}\leq \partial_t u_\eps^\delta (x,t)+h_\eps^\delta(x,t)+ k^{-1},
\end{split}
\]
for $(x,t)\in \mathbb R^n_{j-1}$, and thus, we may estimate the second term on the right hand side of \eqref{eq:small++++} by using $h_\eps^\delta(x,t)$. In particular, in \eqref{eq:use-of-eq}, the expression $\partial_t u_m^-(x,t)- \Phi_0^x(t)-ru_m^-(x,t)$ may be replaced by $h_\eps^\delta(x,t)$. Using these observations, and arguing as in the proof of Lemma \ref{lem:value-for-smooth}, we see that \eqref{eq:small++++agl} becomes
\[
\begin{split}
\mathbb E[&e^{-r\Delta t}u_\delta^\eps( X^{0}(t_1),t_1)]\leq \mathbb E [e^{-2r\Delta t}u_\delta^\eps( X^{1}(t_2),t_2)]\\
&- \mathbb E [\int_{t_1}^{t_2}e^{-r(s-t)}h_{\eps}^\delta(X^{1}(s),s)\ud s ]+c(\Delta t)^{3/2} +k^{-1}\Delta t+c\Delta t\rho(\Delta t),
\end{split}
\]
where $\Delta t=(T_{j-1}-t)/k$, the modulus of continuity $\rho$ in the last error term depends on the modulus of continuity of $h_{\eps}^\delta$, and results from the calculations  similar to those following \eqref{eq:discretization-error}, except that we use $\rho$ instead of $\abs{\cdot}$. Iterating the above reasoning in time, along the lines
of the proof of Lemma \ref{lem:value-for-smooth},  completes the argument. In particular, the last error term yields $c(T_{j-1}-t)\rho((T_{j-1}-t)/k)$. Then, letting $k\to \infty$ \eqref{eq:key-mollif} follows.

Next, we want, for $j$ fixed, to let $\delta\to 0$ and $\eps\to 0$ in \eqref{eq:key-mollif}. To do this, recall that
the underlying dynamics $X$ is defined through the stochastic differential equation in \eqref{eq:dynamics} and based on (uniformly) bounded controls
encoded through ${\str}_m$ and ${\ac}_m$. In particular, $X$ solves a SDE with (uniformly) bounded coefficients. Using this, and a standard martingale argument,  we first observe that there exists, for $\theta>0$ given, $R=R_\theta$ such that
\begin{eqnarray}\label{esta1}
\mathbb P(\sup_{t\le s\le T_{j-1}} \abs{X(s)}\ge R)\le \theta.
\end{eqnarray}
Furthermore, given $\theta>0$ and $R$ as above, we choose $\Om_\theta\subset B_R:=B_R(0)$ such that $\abs{\Om_\theta}<\theta$ and such that
\begin{eqnarray}\label{esta2}
\gamma_{\delta}\to 0\mbox{ uniformly in $(B_R(0)\setminus \Om_{\theta})\times[(j-1)^{-1},T_{j-1}]$ as $\delta\to 0$.}
\end{eqnarray}
Given $E\subset\mathbb R^n$, we let $\chi_E$ denote the indicator function for $E$ in the following. Then, first using \eqref{esta1}, we see that
\[
\begin{split}
&\int_t^{T_{j-1}}\mathbb E[ e^{-r(s-t)}h_\eps^\delta(X(s),s) ]\ud s \\
&\hspace{1 em}\le   \int_t^{T_j} \mathbb E[ h_\eps^\delta(X(s),s)\chi_{B_R}(X(s))]\ud s+\int_t^{T_j} \mathbb E[ h_\eps^\delta(X(s),s)\chi_{B^c_R}(X(s))]\ud s\notag\\
&\hspace{1 em}\leq I_1^{\eps,\delta}(\theta)+I_2^{\eps,\delta}(\theta)+c(T_{j-1}-t) \theta,
\end{split}
\]
for some harmless constant independent of $j$, $\eps$, $\delta$, and $\theta$. Here
\begin{eqnarray}
I_1^{\eps,\delta}(\theta)&:=&\int_t^{T_{j-1}} \mathbb E[ h_\eps^\delta(X(s),s)\chi_{\Om_\theta}(X(s))]\ud s,\notag\\
I_2^{\eps,\delta}(\theta)&:=&\int_t^{T_{j-1}} \mathbb E[ h_\eps^\delta(X(s),s)\chi_{B_R\setminus \Om_{\theta}}(X(s))]\ud s.\notag
\end{eqnarray}

Now
\begin{eqnarray}\label{esta5}
I_1^{\eps,\delta}(\theta)\le c \int_t^{T_{j-1}}  \mathbb E[\chi_{\Om_\theta}(X(s))] \ud s\le c(T_{j-1}-t) \abs{\Om_{\theta}},
\end{eqnarray}
and consequently
\begin{eqnarray}
I_1^{\eps,\delta}(\theta)\le c(T_{j-1}-t) {\theta},\notag
\end{eqnarray}
for some constant $c$ independent of $\eps,\delta,\theta$. Note that the estimate in \eqref{esta5} is far from straightforward. Indeed, \eqref{esta5} is a fundamental estimate by Krylov and stated as Theorem 4 on p.66 in \cite{krylov80}. It is interesting to note that there is, at the core of Krylov's proof of this estimate, a Alexandrov-Bakelman-Pucci-type estimate for uniformly parabolic equations, see Krylov \cite{krylov76}. In particular, a short calculation shows that our dynamics satisfies the sufficient assumptions stated in
\cite{krylov80} for the validity of the estimate in \eqref{esta5}.

Next, we note that
\begin{eqnarray}
I_2^{\eps,\delta}(\theta)\to 0\mbox{ when we first let $\delta\to 0$, and then $\eps\to 0$},\notag
\end{eqnarray}
as we see from the fact that $ h_\eps^\delta=\om(\eps)+\gamma_\delta$, \eqref{esta2}, and that $\om(\eps)\to 0$ as $\eps\to 0$.

Taking the limits $\delta\to 0$ and $\eps\to 0$ in \eqref{eq:key-mollif}, we see by the above that
\[
\begin{split}
u(x,t)\le & \inf_{\rho^-\in{\str}_m}\sup_{A^+\in {\ac}_m} \mathbb E[e^{-r(T_{j-1}-t)}u( X(T_{j-1}),T_{j-1})],
\end{split}
\]
whenever $(x,t)\in \mathbb R^n_{j-1}$. Finally, using the barriers given by Lemma \ref{cor:value-ha}, and by arguing as at the end of the proof of Lemma \ref{lem:equicontinuity} stated below, we can conclude that
\[
\begin{split}
u(x,t)\le &  \inf_{\rho^-\in{\str}_m}\sup_{A^+\in {\ac}_m} \mathbb E[e^{-r(T-t)}g(X(T))],
\end{split}
\]
by letting $j\to \infty$. This completes the proof of Lemma \ref{lem:value}. \hfill $\Box$

\setcounter{equation}{0} \setcounter{theorem}{0}
\section{The limit equation $\partial_t u+F(u,Du,D^2u)=0$}

We here start by introducing the relevant notion of viscosity super- and subsolutions to \eqref{eq:F-eq}.
\begin{definition}\label{def:visc-for-F}
$(a)$ A function $\bar{u}\in \mbox{LSC}_l (\mathbb R^n\times [0,T])$ is a viscosity supersolution to \eqref{eq:F-eq} if $ \bar{u}(x,T)\geq g(x)$ for all $x\in\mathbb R^n$ and if the following holds.  If $(x_0,t_0)\in\mathbb R^n\times (0,T)$ and we have $\phi\in C^{1,2}(\mathbb R^n\times [0,T])$ such that
\begin{eqnarray*}
(i)&&  \bar{u}(x_0, t_0) = \phi(x_0, t_0),\\
(ii)&& \bar{u}(x, t) > \phi(x, t) \mbox{ for } (x,t)\neq (x_0,t_0),
\end{eqnarray*}
then
\begin{eqnarray}\label{aa1tr+}
0\geq \partial_t \phi(x_0,t_0)+F(\bar{u}(x_0,t_0),D\phi(x_0, t_0),D^2\phi(x_0, t_0)),
\end{eqnarray}
whenever $D \phi(x_0,t_0)\neq 0$, and
\begin{eqnarray}\label{aa2tr+}
0\ge \partial_t \phi(x_0,t_0)+\liminf_{p\to 0} F(\bar{u}(x_0,t_0),p,D^2\phi(x_0,t_0)),
\end{eqnarray}
 whenever $D \phi(x_0,t_0)= 0$.

\noindent
$(b)$ A function $\underline u\in \mbox{USC}_l (\mathbb R^n\times [0,T])$ is a viscosity subsolution to \eqref{eq:F-eq} if $\underline u(x,T)\leq g(x)$ for all $x\in\mathbb R^n$ and if the following holds.  If $(x_0,t_0)\in\mathbb R^n\times (0,T)$ and we have $\phi\in C^{1,2}(\mathbb R^n\times [0,T])$ such that
\begin{eqnarray*}
(i)&&  \underline u(x_0, t_0) = \phi(x_0, t_0),\\
(ii)&& \underline u(x, t) < \phi(x, t)\mbox{ for }\ (x,t)\neq (x_0,t_0),
\end{eqnarray*}
then
\begin{eqnarray}\label{bb1tr+}
0\leq \partial_t \phi(x_0,t_0)+F(\underline u(x_0,t_0),D\phi(x_0, t_0),D^2\phi(x_0, t_0)),
\end{eqnarray}
whenever $D \phi(x_0,t_0)\neq 0$, and
\begin{eqnarray}\label{bb2tr+}
0\le \partial_t \phi(x_0,t_0)+\limsup_{p\to 0} F(\underline u(x_0,t_0),p,D^2\phi(x_0,t_0)),
\end{eqnarray}
whenever $D \phi(x_0,t_0)= 0$.

\noindent
$(c)$ If $u$ is both a  viscosity supersolution and a viscosity subsolution to
\eqref{eq:F-eq},  then $u$ is a viscosity solution to \eqref{eq:F-eq}.
\end{definition}


We let
$$F^* :=\limsup_{p\to 0} F\mbox{ and }F_* :=\liminf_{p\to 0} F.$$ Using this notation, we see that \eqref{aa1tr+} and \eqref{aa2tr+} can be written at once as
\[
\begin{split}
0\ge \partial_t \phi(x_0,t_0)+ F_*(\bar{u}(x_0,t_0),D\phi (x_0,t_0),D^2 \phi(x_0, t_0)),
\end{split}
\]
and \eqref{bb1tr+} and \eqref{bb2tr+}  as
\[
\begin{split}
0\le \partial_t \phi(x_0,t_0)+ F^*(\underline u(x_0,t_0),D\phi (x_0,t_0),D^2 \phi(x_0, t_0)).
\end{split}
\]
Similarly to Lemma \ref{lem:comparison-m}, the following lemma also follows from \cite{gigagis91}.
\begin{lemma} \label{thm:value-u-a}
 Let $\underline{u}, \bar u\in \mbox{C}_l(\mathbb R^n\times [0,T])$  be viscosity sub- and supersolutions to \eqref{eq:F-eq} in the sense of Definition \ref{def:visc-for-F}.
Then,
$$\underline{u}(x,t)\leq \bar{u}(x,t),$$
whenever $(x,t)\in \mathbb R^n\times[0,T]$.
\end{lemma}

  Furthermore, arguing as in Lemma \ref{thm:value-ha} and Lemma \ref{cor:value-ha}, we see that we can construct barriers to \eqref{eq:F-eq}, use them for comparison, and prove the following lemma.
\begin{lemma} \label{thm:value-hau}
Let $y\in\mathbb R^n$ and let $L$ be the Lipschitz constant of $g$. Consider $0<\ep\ll 1$, and let
\begin{eqnarray*}
\bar{w}(x,t)=g(y)+\frac A{\ep^2}(T-t)+2L(|x-y|^2+\ep)^{1/2},\notag\\
\underline{w}(x,t)=g(y)-\frac A{\ep^2}(T-t)-2L(|x-y|^2+\ep)^{1/2}.
\end{eqnarray*}
Then, we can choose $A$, independent of $y$, $\ep$ and $m$, so that
$\bar{w}$ and $\underline{w}$ are viscosity super- and subsolutions to
\eqref{eq:F-eq}. Consequently, for such $A$, and if $u$ is a viscosity solution to \eqref{eq:F-eq}, we have
$$
\underline{w}\leq u\leq \bar{w}.
$$
\end{lemma}
Below, we always choose, when applying $\bar{w}$ and $\underline{w}$, $A$ so that Lemma \ref{thm:value-hau} holds.

\begin{theorem}\label{th2-}
 There exists a unique viscosity solution $u$ to \eqref{eq:F-eq}.
\end{theorem}

The uniqueness part of Theorem \ref{th2-} follows from Lemma \ref{thm:value-u-a}.  The existence part of Theorem \ref{th2-} again follows from Perron's method, also using Lemma \ref{thm:value-hau} as discussed after Lemma \ref{thm:value-haha}.


Next, using  a modification of the techniques \cite{chengg91}, \cite{evanss91}, see also \cite{juutinenk06} and \cite{kawohlmp12}, we prove the following lemma which states that the set of test functions used in Definition \ref{def:visc-for-F} can be reduced. We consider continuous sub-/supersolutions as we later will only need this result  for solutions.
\begin{lemma}\label{lem:restricted-class}
Let $u\in \mbox{C}_l (\mathbb R^n\times [0,T])$. Then, to test whether or not $u$ is a viscosity super- or subsolution at $(x_0,t_0)$ in the sense of Definition~\ref{def:visc-for-F}, it is enough
to consider test functions $\phi\in C^{1,2}(\mathbb R^n\times [0,T])$ such that either
\begin{eqnarray*}
(i)&&D \phi(x_0,t_0)\neq 0\mbox{ or }\\
(ii)&&D \phi(x_0,t_0)=0\mbox{ and } D^2 \phi(x_0,t_0)=0.
\end{eqnarray*}
\end{lemma}
\begin{proof} We here only prove the lemma in the context of subsolutions. The proof is by contradiction. Indeed, assume that there exists a function $u\in \mbox{C}_l (\mathbb R^n\times [0,T])$, {\it which fails to be a subsolution at $(x_0,t_0)$} in the sense of Definition \ref{def:visc-for-F}  even though
the following holds. If $(x_0,t_0)\in\mathbb R^n\times (0,T)$ and $\phi\in C^{1,2}(\mathbb R^n\times [0,T])$ are such that
\begin{eqnarray} 
(i)&&  u(x_0, t_0) = \phi (x_0, t_0),\notag \\
(ii)&& u(x, t) < \phi (x, t)\mbox{ for } (x,t)\neq (x_0,t_0)\notag,
\end{eqnarray}
then
\begin{eqnarray}\label{tr2}
0\le \partial_t \phi(x_0,t_0)+ F^*(u(x_0,t_0),D\phi(x_0,t_0),D^2\phi(x_0,t_0)),
\end{eqnarray}
whenever
\begin{eqnarray}\label{tr3}
(i)&&D \phi(x_0,t_0)\neq 0\mbox{ or }\notag\\
(ii)&&D \phi(x_0,t_0)=0\mbox{ and } D^2 \phi(x_0,t_0)=0.\notag
\end{eqnarray}
Using that $u$ is assumed to fail to be a subsolution, we see that there must also exist a test function $\varphi$ touching from above, and $\eps>0$, such that
\begin{eqnarray}
\label{eq:CP}
0>\partial_t \varphi(x_0,t_0)+ F^*(u(x_0,t_0),D\varphi(x_0,t_0), D^2\varphi(x_0,t_0))+\eps,
\end{eqnarray}
and such that
\begin{eqnarray}
\label{eq:CP+}
D \varphi(x_0,t_0)=0 \text{ and } D^2 \varphi(x_0,t_0)\neq 0.\notag
\end{eqnarray}
In addition, we may assume 
that $u-\varphi$ has a strict global maximum at $(x_0,t_0)$. Let
\begin{eqnarray}
w(x,t,y,s):=w_j(x,t,y,s)=u(x,t)-\varphi(y,s)-\Psi_j(x,t,y,s),
\end{eqnarray}
where
\[
\begin{split}
\Psi(x,t,y,s):=\Psi_j(x,t,y,s)=\frac{j}{4}\abs{x-y}^4+\frac{j}{2}(t-s)^2.
\end{split}
\]
By comparison and the structure of the barriers in Lemma~\ref{thm:value-hau}, we see that there exists  $( x_j,  t_j, y_j, s_j)\in \mathbb R^n \times (0,T) \times \mathbb R^n \times (0,T)$ such that
\begin{equation}\label{comp9}
w ( x_j,  t_j, y_j, s_j) = \sup_{(x,t,y,s) \in \mathbb R^n \times [0,T] \times \mathbb R^n \times [0,T]} w(x,t,y,s).
\end{equation}
Furthermore,
\begin{equation*}\label{comp9pp}
( x_j,  t_j, y_j, s_j)\to(x_0,t_0,x_0,t_0)\quad \mbox{ as }\quad j\to \infty.
\end{equation*}

We now consider two cases.  \\
\noindent {\bf Case 1:} there exists an
infinite sequence of $j$:s such that $x_{j}= y_{j}$ for each such $j$. \\
\noindent {\bf Case 2:} there exists an $j_0\in \{1,2,\ldots\}$ such that $x_{j}\neq  y_{j}$ for all $j$, $j>j_0$.

We first analyze Case 1 and we let $x_{j}=y_{j}$. Then, by construction,
\begin{eqnarray*}
u(x,t)&\leq& u( x_{j},  t_{j})+\varphi(y,s)-\varphi( y_{j}, s_{j})\notag\\
&&+\Psi(x,t,y,s)-\Psi( x_{j},  t_{j}, y_{j}, s_{j}),
\end{eqnarray*}
whenever $(x,t,y,s) \in \mathbb R^n \times [0,T] \times \mathbb R^n \times [0,T]$. In particular,
\begin{eqnarray}\label{testf1}
u(x,t)\leq u( x_{j},  t_{j})+\Psi(x,t, y_{j}, s_{j})-\Psi( x_{j},  t_{j}, y_{j}, s_{j}),
\end{eqnarray}
whenever $(x,t) \in \mathbb R^n \times (0,T) $. Moreover, as $ x_{j}= y_{j}$, we  observe
that $D_x\Psi( x_{j},  t_{j}, y_{j}, s_{j})=0$, $D_{xx}^2 \Psi( x_{j},  t_{j}, y_{j}, s_{j})=0$, and thus, we can conclude that the function on the right in \eqref{testf1} is an admissible test function at $( x_{j},  t_{j})$
for the conclusion in \eqref{tr2}. For brevity, we drop the arguments  $(x_{j},  t_{j}, y_{j}, s_{j})$ in $\Psi$ and its derivatives in the displays below.
We have
\[
\begin{split}
0\le \Psi_t &+F^*(u( x_{j},  t_{j}),D_x\Psi,D^2_{xx} \Psi).
\end{split}
\]
Using that $(y,s)\mapsto \varphi(y,s)+\Psi( x_{j},  t_{j},y,s)$ has a local minimum at $( y_{j}, s_{j})$ by \eqref{comp9}, we observe that $0=-D^2_{yy} \Psi \le D^2 \varphi( y_{j}, s_{j})$, $0=-D_y \Psi=D \varphi( y_{j}, s_{j})$ and  $-\Psi_s =\partial_t \varphi( y_{j}, s_{j})$. By these facts, ellipticity of $F^*$, and \eqref{eq:CP} it follows that
\[
\begin{split}
\Psi_s=-\partial_t \varphi(y_{j},s_{j})> F^*(u(y_{j},s_{j}),-D_y\Psi, -D^2_{yy}\Psi)+\frac\eps2.
\end{split}
\]
\sloppy
Combining the previous two displays, we obtain
\[
\begin{split}
r(u( x_{j},  t_{j})-u(y_{j},s_{j}))&<j\big((t_{j}-s_{j})-(t_{j}-s_{j})\big)-\frac\eps2=-\frac\eps2,
\end{split}
\]
which is a a contradiction for large enough $j$ as $u$ is continuous.

We next analyze Case 2. In this case, using the Theorem of sums, see \cite{crandallil92}, we see that there exist $(\Psi_t,D_x\Psi,X)\in \ol{\mathcal P}^{2,+}u( x, t)$ and $(-\Psi_s,-D_y \Psi,-Y)\in \ol{\mathcal P}^{2,-}\varphi( y, t)$  such that $X\le -Y$. In particular, as $D_x\Psi\neq 0$ it follows that
\begin{equation}
\label{eq:crucial-ineqs}
\begin{split}
0 &\le \Psi_t+F^*(u( x_{j}, t_{j}),D_x \Psi,X)\\
0&>-\Psi_s +F^*(u( y_{j},s_{j}),-D_y\Psi,-Y)+\frac\eps2,\\
\end{split}
\end{equation}
where the second inequality follows from \eqref{eq:CP}. 
Thus
\[
\begin{split}
\Psi_t+\Psi_s &> -F^*(u( x_{j}, t_{j}),D_x\Psi,X)+F^*(u( y_{j},s_{j}),-D_y\Psi,-Y)+\frac\eps2\\
&\ge \frac\eps2+r(u( y_{j},s_{j})-u( x_{j}, t_{j})).
\end{split}
\]
Finally, observing that $\Psi_t=-\Psi_s$, we see that the last display implies that
\[
\begin{split}
0 &\ge \frac\eps2+r(u( y_{j},s_{j})-u( x_{j}, t_{j})),
\end{split}
\]
and this now produces a contradiction for $j$ large enough. The proof for a supersolution is similar.
\end{proof}

\setcounter{equation}{0} \setcounter{theorem}{0}

\section{Going to the limit: general action sets as $m\to\infty$}

\label{sec:limits}

In the following we will use the following lemma.
\begin{lemma}\label{lem:approxham-} Let 
$\xi_m,\xi\in\mathbb R$, $p_m,p\in\mathbb R^n\setminus\{0\}$, $M_m,M\in\mathcal{M}(n)$, be such that
$$
\xi_m\to \xi,\quad p_m\to p, \quad \text{and}\quad  M_m\to M,
$$
as $m\to\infty$.
Then,
\[
\begin{split}
H_m^\pm(\xi_m,p_m,M_m)
\to -F(\xi,p,M).
\end{split}
\]
\end{lemma}
\begin{proof} We here only prove the statement for $H_m^-$,  the proof for $H^+_m$ being analogous. To prove that $$H_m^-(\xi_m,p_m,M_m)\to -F(\xi,p,M)$$ it suffices to consider those terms  in $H_m^-$ which are sensitive to $$\inf_{(\theta^+,d^+)\in\mathcal{H}_m}\ \ \sup_{(\theta^-,d^-)\in\mathcal{H}_m}.$$ In particular, we focus on
\begin{eqnarray*}
\tilde \Phi(\theta^+,\theta^-,d^+,d^-,p_m,M_m)&:=&-\half (\theta^+-\theta^-)'\Sigma M_m\Sigma (\theta^+-\theta^-)\notag\\
&&-(d^++d^-)(\theta^++\theta^-)\cdot p_m.
\end{eqnarray*}
Setting
\[
\begin{split}
\tilde \Phi_m 
&:=\inf_{(\theta^+,d^+)\in\mathcal{H}_m}\ \ \sup_{(\theta^-,d^-)\in\mathcal{H}_m}\tilde \Phi,
\end{split}
\]
we observe that
\[
\begin{split}
\tilde\Phi_m&\leq \sup_{(\theta^-,d^-)\in\mathcal{H}_m} \Big(-\half (p_m/|p_m|-\theta^-)'\Sigma M_m\Sigma (p_m/|p_m|-\theta^-)\notag\\
&\hspace{9 em}-d^-(p_m/|p_m|+\theta^-)\cdot p_m\Big),
\end{split}
\]
and that $(p_m/|p_m|+\theta^-)\cdot p_m\geq 0$ whenever $\theta^-\in \mathbb S^{n-1}$. In particular, we conclude that
$\tilde \Phi_m$ is bounded from above as $m\to \infty$.

Using that the set $\{(\theta^+,d^+)\in\mathcal{H}_m\}$ is compact, we see that there exists a $(\theta^+_{m},d^+_{m})$ realizing the infimum in the definition of
$\tilde \Phi_m$. We will prove
that there exists, given $\eps>0$, $m_0=m_0(\eps)$ such that
\begin{eqnarray}\label{ml3}
|p_m|-\eps\leq\theta_m^+\cdot p_m\leq |p_m|\qquad \mbox{whenever}\quad m\geq m_0.
\end{eqnarray}
Obviously, we only have to establish the lower bound and to do this we assume, on the contrary, that there exists $\eps>0$
and $m_j\to \infty$, such that
\begin{eqnarray}\label{ml3-}
\theta_{m_j}^+\cdot p_{m_j}\leq |p_{m_j}|- \ep\mbox{ as $j\to\infty$}.
\end{eqnarray}
If this is the case, then
\begin{equation}
\label{eq:large-lower-bound}
\begin{split}
\tilde \Phi_{m_j}&=\sup_{(\theta^-,d^-)\in\mathcal{H}_{m_j}}\biggl ( -\half (\theta_{m_j}^+-\theta^-)'\Sigma M_{m_j}\Sigma (\theta_{m_j}^+-\theta^-)\\
&\hspace{9 em}-(d_{m_j}^++d^-)(\theta_{m_j}^++\theta^-)\cdot p_{m_j}\biggr )\\
&\ge -c-(d_{m_j}^++{m_j})(\theta_{m_j}^+-p_{m_j}/|p_{m_j}|)\cdot p_{m_j}\\
&\geq -c+(d_{m_j}^++{m_j}) \eps,
\end{split}
\end{equation}
as $(-p_{m_j}/|p_{m_j}|,{m_j})\in\mathcal{H}_{m_j}$ and for some harmless constant $c$.
However, \eqref{eq:large-lower-bound} contradicts the boundedness of $\tilde \Phi_{m_j}$  as
$m_j\to\infty$, hence \eqref{ml3-} must be false and \eqref{ml3} must hold.

Using \eqref{ml3} we see that
\begin{eqnarray}\label{ml3++}
\theta_m^+\to p/|p|\mbox{ as $m\to\infty$}.
\end{eqnarray}
Furthermore, using that
\[
\begin{split}
\tilde \Phi_m&\geq-\half (\theta_m^++p_m/|p_m|)'\Sigma M_m\Sigma (\theta_m^++p_m/|p_m|)\notag\\
&\hspace{1 em}-(d_m^++m)(\theta_m^+-p_m/|p_m|)\cdot p_m\notag\\
&\geq -\half (\theta_m^++p_m/|p_m|)'\Sigma M_m\Sigma (\theta_m^++p_m/|p_m|),
\end{split}
\]
in combination with \eqref{ml3++}, we have that
\begin{equation}
\label{eq:liminf-lower-bound}
\begin{split}
\liminf_{m\to\infty}\tilde \Phi_m&\geq-2(p/|p|)'\Sigma M\Sigma p/|p|.
\end{split}
\end{equation}
This yields, recalling the rest of the terms in the definition of  $H^-_m$, that
\[
\begin{split}
\liminf_{m\to\infty} H^-_m(\xi_m,p_m,M_m)\geq  -F(\xi,p,M).
\end{split}
\]

To complete the proof it only remains to prove that
\begin{equation}
\label{eq:limsupbound}
\begin{split}
 \limsup_{m\to\infty} H^-_m(\xi_m,p_m,M_m)\leq -F(\xi, p,M).
\end{split}
\end{equation}
To do this, we first note, again using the definition of $(\theta^+_{m},d^+_{m})$, that
\[
\begin{split}
\tilde \Phi_m&= \sup_{(\theta^-,d^-)\in\mathcal{H}_m}\tilde\Phi(\theta_m^+,\theta^-,d_m^+,d^-,p_m,M_m)\\
&\le \sup_{(\theta^-,d^-)\in\mathcal{H}_m}\tilde\Phi(p_m/|p_m|,\theta^-,m,d^-,p_m,M_m).
\end{split}
\]
Furthermore, using compactness, we see that we can choose $(\theta^-_m,d^-_m)$ realizing the supremum in the last display. Hence,
\begin{eqnarray*}
\tilde \Phi_m&\le& -\half (p_m/|p_m|- \theta^-_m)'\Sigma M_m\Sigma (p_m/|p_m|- \theta^-_m) \notag\\
&&-(m+ d_m^-)(p_m/|p_m|+ \theta^-_m)\cdot p_m.
\end{eqnarray*}
Using this we deduce that $ \theta^-_m\to -p/|p|$, as otherwise the above estimate would imply $\liminf_{m\to\infty}\tilde \Phi_m=-\infty$ contradicting \eqref{eq:liminf-lower-bound}. Furthermore,
\[
\begin{split}
\tilde \Phi_m \le& -\half (p_m/|p_m|- \theta^-_m)'\Sigma M_m\Sigma (p_m/|p_m|- \theta^-_m)\\
&-(m+ d_m^-)(p_m/|p_m|+ \theta^-_m)\cdot p_m\\
\le& -\half (p_m/|p_m|- \theta^-_m)'\Sigma M_m\Sigma (p_m/|p_m|- \theta^-_m),
\end{split}
\]
and taking $\limsup_{m\to \infty}$, we see that \eqref{eq:limsupbound} holds. This completes the proof of the lemma.
\end{proof}


\begin{lemma} \label{lem:equicontinuity}  Let $u_m^+$ and $u_m^-$ be the unique solutions to \eqref{par9ap-} and \eqref{par10ap-}, respectively, ensured
by Lemma \ref{thm:value-haha}. Then, there exists $m_0\in\{1,2,\ldots\}$ such that the families
$$
\{u_m^\pm:\ m\geq m_0\}
$$
are equicontinuous on $\mathbb R^n\times [0,T]$.
\end{lemma}
\begin{proof}  We here only prove that $\{u_m^+:\ m\geq m_0\}$ is equicontinuous on $\mathbb R^n\times [0,T]$, the proof for $\{u_m^-:\ m\geq m_0\}$ being analogous. Using Lemma \ref{lem:value} we have that
\begin{eqnarray*}
u_m^+(x,t)=U_m^+(x,t):=\sup_{\rho^+\in{\str}_m}\inf_{A^-\in {\ac}_m}J^{(x,t)}(\rho^+(A^-),A^-),
\end{eqnarray*}
whenever  $(x,t)\in \mathbb R^n\times [0,T]$. Note that we can use both a stochastic as well as a PDE point of view to prove the lemma. Furthermore, the processes underlying the stochastic formulation, see \eqref{eq:dynamics}, all end at $T$. Suppose that we consider two games, one starting from $(x_1,t_1)$ and one starting from $(x_2,t_2)$ with  $t_1<t_2$. We want to show, uniformly in $m$, that $|u_m^+(x_1,t_1)-u_m^+(x_2,t_2)|$ can be made arbitrary small by considering $(x_1,t_1)$ and $(x_2,t_2)$ sufficiently close. Using that the controls and strategies always can be 'copied' for the processes starting from  $(x_1,t_1)$, $(x_2,t_2)$, cf. p.\ 105 \cite{PS}, and by considering same samples, this is possible as we use space and time independent Brownian motions, we see that it is enough to consider points $(x_1,t_1)$ and $(x_2,T)$ with $t_1<T$.
In particular, given $ \delta>0$, we now want to prove that there exists $\eta>0$ such that
\[
\begin{split}
|u_m^+(x_1,t_1)-u_m^+(x_2,T)|=|u_m^+(x_1,t_1)-g(x_2)|\leq \delta
\end{split}
\]
whenever $|x_1-x_2|+T-t_1\leq \eta$.
Recall the barriers
\[
\begin{split}
\bar{w}(x,t)=g(x_2)+\frac A{\ep^2}(T-t)+2M(|x-x_2|^2+\ep)^{1/2},\notag\\
\underline{w}(x,t)=g(x_2)-\frac A{\ep^2}(T-t)-2M(|x-x_2|^2+\ep)^{1/2}.
\end{split}
\]
Using Lemma \ref{cor:value-ha}, we have
\[
\begin{split}
\underline{w}\leq u_m^+\leq \bar{w}.
\end{split}
\]
 In particular,
\[
\begin{split}
|u_m^+(x_1,t)-g(x_2)|\leq \frac A{\ep^2}(T-t)+2M(|x_1-x_2|^2+\ep)^{1/2}.
\end{split}
\]
Let  $|x_1-x_2|+T-t_1\leq\ep^{5/2}$. Then, for $\ep<1$
\[
\begin{split}
|u_m^+(x_1,t_1)-g(x_2)|\leq A\ep^{1/2}+4M\ep^{1/2}\leq (A+4M)\ep^{1/2},
\end{split}
\]
and we conclude, by choosing $\eps$ small enough. 
This completes the proof.
\end{proof}


\begin{lemma}\label{lemimportant}  Let $u_m^+$ and $u_m^-$ be the unique solutions to \eqref{par9ap-} and \eqref{par10ap-}, respectively.  Then,
$$
u_m^\pm(x,t)\to u(x,t),$$
 for all $(x,t)\in \mathbb R^n\times [0,T]$,
where $u$ is the continuous, unique viscosity solution to \eqref{eq:F-eq}.
\end{lemma}
\begin{proof}
We again only prove the result for $u_m^+$, the proof for $u_m^-$ being similar. We first recall that  the existence of $u_m^+$ is ensured
by Lemma \ref{thm:value-haha}. Furthermore, by comparison with a supersolution $L$, we see that the sequence $\{u_m^+\}$ is uniformly bounded in $\mathbb R^n\times [0,T]$.  Using this and Lemma \ref{lem:equicontinuity}, we can first conclude, using the Arzel\`a-Ascoli theorem, that there exists $u$, continuous
on $\mathbb R^n\times [0,T]$, such that
\begin{eqnarray}
u_m^+(x,t)\to u(x,t)\quad \mbox{ as $m\to \infty$}.
\end{eqnarray}

We next prove that $u$ is a viscosity subsolution in $\mathbb R^n\times [0,T)$ to \eqref{eq:F-eq}.  To do this, let $\phi\in C^2$ touch $u$ strictly from above at $(x_0,t_0)$. Then, using the uniform convergence it follows that there exists $(x_m,t_m)\to (x_0,t_0)$ such that
\[
\begin{split}
u_m^+-\phi
\end{split}
\]
has a strict max at $(x_m,t_m)$. Hence
\begin{equation}\label{meq}
\partial_t \phi(x_m,t_m)\ge H^+_m(u_m^+(x_m,t_m),D\phi(x_m,t_m),D^2\phi(x_m,t_m)).
\end{equation}
 Note that, as $m\to\infty$, $\partial_t \phi(x_m,t_m)\to\partial_t \phi(x_0,t_0)$, $u_m^+(x_m,t_m)\to u(x_0,t_0)$, $D\phi(x_m,t_m)\to D\phi(x_0,t_0)$, $D^2\phi(x_m,t_m)\to D^2\phi(x_0,t_0)$, and we want to pass to the limit in \eqref{meq}. Suppose first that $D\phi(x_0,t_0)\neq 0$. Then, using Lemma \ref{lem:approxham-} we see that
\[
\begin{split}
H^+_m (u_m^+(x_m,t_m),&D\phi(x_m,t_m),D^2\phi(x_m,t_m))\\
&\to -F(u(x_0,t_0),D\phi(x_0,t_0),D^2\phi(x_0,t_0)),
\end{split}
\]
as $m\to \infty$. Next, suppose that $D\phi(x_0,t_0)=0$. In this case, we can, by Lemma \ref{lem:restricted-class}, also assume, without loss of generality, that $D^2\phi(x_0,t_0)=0$. But in this case
\[
\begin{split}
H_m^+ (u_m^+(x_m,t_m),&D\phi(x_m,t_m),D^2\phi(x_m,t_m))\\
&\to -F^*(u(x_0,t_0),0,0).
\end{split}
\]
In particular, in either case, we can conclude that
\begin{equation}\label{meq+}
\partial_t \phi(x_0,t_0)\ge  -F^*(u(x_0,t_0),D\phi(x_0,t_0),D^2\phi(x_0,t_0)).
\end{equation}
and hence $u$ is a continuous viscosity subsolution to \eqref{eq:F-eq}. The proof of the result that $u$ is also a supersolution to \eqref{eq:F-eq} is similar. We omit further details.\end{proof}

We are in position to prove the main result of the paper, Theorem \ref{th2}, which states that the game with unbounded controls has a value and that
the value function
\begin{eqnarray*}
u=U^+(x,t)&=&\sup_{\rho^+\in{\str}}\inf_{A^-\in {\ac}}J^{(x,t)}(\rho^+(A^-),A^-)\\
&=&\inf_{\rho^-\in{\str}}\sup_{A^+\in {\ac}}J^{(x,t)}(A^+,\rho^-(A^+))=U^-(x,t),
\end{eqnarray*}
is the unique solution $u$ to \eqref{eq:F-eq}.

\begin{proof}[Proof of Theorem \ref{th2}] We will here only prove that $u=U^-$ as the proof is analogously in the other case. Recall that Lemma \ref{lem:value} states that
\[
\begin{split}
&u_m^-(x,t)=U_m^-(x,t)=\inf_{\rho^-\in{\str}_m}\sup_{A^+\in {\ac}_m}J^{(x,t)}(A^+,\rho^-(A^+)).
\end{split}
\]
Furthermore, using Lemma \ref{lemimportant} we have
\[
\begin{split}
u^-_m(x,t)\to u(x,t),
\end{split}
\]
where $u$  is the solution  to \eqref{eq:F-eq}. Thus, it suffices to prove that
\begin{equation}\label{meq++}
U_m^-(x,t)\to U^-(x,t)\mbox{ as $m\to\infty$}.
\end{equation}
To prove \eqref{meq++}, we first note that
\[
\begin{split}
U^-(x,t)&=\inf_{\rho^-\in{\str}}\sup_{A^+\in {\ac}}J^{(x,t)}(A^+,\rho^-(A^+))\\
&\geq \inf_{\rho^-\in{\str}}\sup_{A^+\in {\ac}_m}J^{(x,t)}(A^+,\rho^-(A^+)).
\end{split}
\]
In particular, given $\ep>0$, using that $U^-(x,t)$ is finite by our assumptions on $g$ and that ${\str}=\cup_m{\str}_m$, we see that there exists
$m_0=m_0(\ep)$ such that
\[
\begin{split}
U^-(x,t)\geq \inf_{\rho^-\in{\str}}\sup_{A^+\in {\ac}_m}J^{(x,t)}(A^+,\rho^-(A^+))\geq U_m^-(x,t)-\ep,
\end{split}
\]
whenever $m\geq m_0$. We can therefore conclude that
\[
\begin{split}
U^-(x,t)\geq \limsup_{m\to \infty}U_m^-(x,t).
\end{split}
\]
To complete the proof of \eqref{meq++}, it hence only remains to prove that
\begin{eqnarray}\label{impa}
U^-(x,t)\leq \liminf_{m\to \infty}U_m^-(x,t).
\end{eqnarray}

To prove \eqref{impa}, we first fix a strategy which estimates the infimum when the supremum is taken over the controls $\ac_k$. Then, by choosing $k$ large enough, we can closely estimate the original supremum taken over $\ac$ by a supremum taken over $\ac_k$. To write down the details,  we recall that
\[
\begin{split}
{\ac}_k&:=\{A\in\ac:\ {\Lambda}(A)\leq k\},\notag\\
{\str}_m&:=\{\rho\in \str:\ \Lambda(\rho)\leq  m\},
\end{split}
\]
for  $k=1,2,\ldots$.
Fix $\ep>0$. For each $k$, we choose $\rho_{km}^-\in {\str}_m $  such that
\begin{eqnarray}
\label{ii1+uu}
&&\sup_{A^+\in {\ac}_k }J^{(x,t)}(A^+,\rho_{km}^-(A^+))\notag\\
&&\leq \inf_{\rho^-\in {\str}_m }\sup_{A^+\in {\ac}_k }J^{(x,t)}(A^+,\rho^-(A^+))+\ep.
\end{eqnarray}
Next, we define
$$\rho_m^-(A^+):=\rho_{km}^-(A^+)\mbox{ whenever $A^+\in {\ac}_k \setminus {\ac}_{k-1} $},$$
and we set ${\ac}_{0} =\emptyset$ in order to get started. Now, using that
${\ac} =\cup_k{\ac}_k $ we  have, for $k\ge k_\eps$ sufficiently large,  that
\[
\begin{split}
U^-(x,t)&\leq \sup_{A^+\in {\ac} }J^{(x,t)}(A^+, \rho_m^-(A^+))\notag\\
&\leq\sup_{A^+\in {\ac}_k }J^{(x,t)}(A^+,\rho_m^-(A^+))+\ep\notag\\
&\leq \inf_{\rho^-\in {\str}_m }\sup_{A^+\in {\ac}_k }J^{(x,t)}(A^+,\rho^-(A^+))+2\ep,
\end{split}
\]
where we on the last line have used \eqref{ii1+uu}. Assuming  $m\ge k_\eps$, we can choose $k=m$ in the last display. Hence we can conclude that there exists, given $\ep>0$,  $m_0=m_0(\ep)$ such that if $m\geq m_0$, then
\begin{eqnarray*}
U^-(x,t)&\leq& \inf_{\rho^-\in {\str}_m }\sup_{A^+\in {\ac}_m }J^{(x,t)}(A^+,\rho^-(A^+))+2\ep\notag\\
&=&U_m^-(x,t)+2\ep.
\end{eqnarray*}
This proves \eqref{impa}. \end{proof}

\def\cprime{$'$} \def\cprime{$'$}

\end{document}